\documentclass[11pt]{article}
\usepackage{amsmath,amsfonts,amsthm}
\usepackage[margin=1in]{geometry}
\usepackage[normalem]{ulem}
\usepackage{amssymb}
\usepackage{appendix}

\usepackage[hyphens]{url}
\usepackage{hyperref}
\usepackage{graphicx}
\usepackage{epstopdf}
\usepackage{float}
\usepackage{subfigure}
\usepackage{verbatim}
\usepackage{enumitem}
\usepackage{color}
\usepackage{bm}
\usepackage{dsfont}

\usepackage{soul}

\numberwithin{equation}{section}
\theoremstyle{plain}

\newcommand{\tU}{\tilde{U}}

\newcommand{\bh}{\bar{h}}
\newcommand{\hH}{\hat{H}}
\newcommand{\uh}{\underline{h}}

\newcommand{\by}{\bar{y}}
\newcommand{\uy}{\underline{y}}

\newtheorem{lemma}{Lemma}[section]
\newtheorem{theorem}{Theorem}[section]

\newtheorem{proposition}{Proposition}[section]
\newtheorem{corollary}{Corollary}[section]
\newtheorem{assumption}{Assumption}[section]
\newtheorem{remark}{Remark}[section]
\allowdisplaybreaks[1]

\def \E{\mathbb{E}}

\makeatletter
\def\@setcopyright{}
\def\serieslogo@{}
\makeatother

\title{Optimal consumption under adjustment costs with respect to multiple reference levels}
\author{
Yijie HUANG\thanks{Email: huang1@mail.ustc.edu.cn, School of Mathematical Sciences, University of Science and Technology of China, Hefei 230026, China.} \and
Kaixin YAN\thanks{Email: kaixinyan@stu.xmu.edu.cn, School of Mathematical Sciences, Xiamen University, Xiamen 361005, China.} \and
Qinyi ZHANG\thanks{Email: qinyi16.zhang@polyu.edu.hk, Department of Applied Mathematics, The Hong Kong Polytechnic University, Kowloon, Hong Kong}
}

\begin{document}

\date{\vspace{-0.7cm}}
\maketitle

\begin{abstract}
This paper studies a type of consumption preference where some adjustment costs are incured whenever the past spending maximum and the past spending minimum records are updated. This preference can capture the adverse effects of the historical consumption high and low values on the agent's consumption performance, thereby matching with some empirically observed smooth consumption patterns. By employing the dual transform, the smooth-fit conditions and the super-contact conditions, we obtain the closed-form solution of the dual PDE problem, and can characterize the optimal investment and consumption controls in the piecewise feedback form.  We provide the rigorous proof of the verification theorem and compensate the theoretical findings with some numerical examples and financial implications.
\vspace{0.1in}


\end{abstract}

\section{Introduction}\label{sec: intro}
Starting from Merton's seminal studies \cite{Mert1969RES, Mert1971JET}, the optimal consumption-investment problem via utility maximization has been one of the core research topics in quantitative finance. Over the past half-century, researchers have delved into numerous extensions to accommodate more general and complex financial market models, trading constraints, and other factors in decision making. 
This exploration has led to the rapid development of a plethora of influential results and methodologies.
Various extensions can be found in the literature such as the stochastic income (\cite{Wang2009JME}, \cite{Zeldes1989QJE}), the stochastic differential utility (\cite{SchroderSkiadas1999JET}), transaction costs (\cite{Liu2004JoF}), and wealth constraints (\cite{ElieTouzi2008FS}, \cite{JeonKwak2024FS}). 

Some empirical studies (see, for instance, \cite{Campbell1989}) have posited that the observed consumption tends to exhibit over-smoothness, which cannot be explained by the optimal solutions induced by some standard time-separable utility functions.  To partially explain the smooth consumption behavior, existing studies suggest incorporating the past consumption decisions into the utility function.
By introducing a reference level (either exogenous or endogenous), the optimal solution becomes smoother with fewer fluctuations. 
Time-nonseparable preferences have gained traction in modeling consumption patterns due to their ability to account for the observed consumption smoothness and the equity premium puzzle. 

In the existing literature, different means have been considered to incorporate the historical dependence into the utility function to measure the consumption performance. The most well-known one is the linear habit formation preference that concerns the difference between the current consumption rate and the habit formation process that is defined as the weighted average of the past consumption plan, see \cite{Ango22}, \cite{Ango23}, \cite{Constantinides1990JPE}, \cite{DetempleZapatero1992MF}, \cite{EnglezosKaratzas2009Sicon}, \cite{SchroderSkiadas2002RFS}, \cite{Bilsen}, \cite{YangYu2019}, \cite{Yu2015AoAP}, among others. The multiplicative habit formation has also gained a lot of attention, which generates the utility on the ratio of consumption rate and the habit formation process; see, \cite{Abel1990AER}, \cite{AngoYuY2024ArXiv}, \cite{Carroll}, \cite{Corr}, \cite{VanB}, to name a few. Another stream of research consider the historical dependence as the past running maximum of the consumption process, see among  \cite{AngoshtariBayraktarY2019Sifin}, \cite{DengLiPY2022FS}, \cite{GuasoniHubermanR2020MFF}, \cite{LiYuZ2023IME}, \cite{LiYuZ2024SiFin}.

Recently, \cite{ChoiJeonKoo2022JET} proposes a different perspective on the relative consumption by considering the adjustment costs.
Instead of directly incorporating the reference level into the utility function, they separate the reference level from consumption by introducing an additional cost term.
As a result of the adjustment costs arising from changes in the reference level,  the optimal consumption therein is relatively smooth and can partially account for the smooth consumption behavior observed in the market. Along this direction, several interesting extensions have been considered in some subsequent studies in \cite{JeonJ}, \cite{JeonJ2} and \cite{JeonJ1}.

Considering a reference level is an effective way to explain smooth consumption behavior. Different reference levels capture distinct features of past consumption. 
However, the information conveyed by a single reference level is always limited. If the influence of two or more reference levels on optimal consumption can be considered simultaneously, the results are likely to be more compelling and nuanced. 
Meanwhile, regardless of how the influence of the reference level is characterized in the utility, existing literature only considers single references (such as the habit formation, the past spending maximum, or the current consumption level). 

In this paper, we propose and study a type of consumption preference with multiple reference levels, namely, both the past consumption running maximum and the past consumption running minimum. 
Following the representations of \cite{ChoiJeonKoo2022JET}, we separate consumption from the reference level in the utility function, considering the adjustment costs that incur whenever the past consumption maximum level increases or the past consumption minimum level decreases. The incorporation of two adjustment cost terms will naturally enforce the optimal consumption behavior to be smooth.
Compared with \cite{ChoiJeonKoo2022JET}, our adjustment reference is chosen as the past consumption maximum and minimum levels instead of the current consumption level,  allowing for possibly larger and more flexible consumption fluctuations to achieve the higher expected utility on consumption. 
In solving the problem, we utilize a dynamic programming approach to derive the associated Hamilton-Jacobi-Bellman (HJB) equation. 
By employing dual transformation, smooth-fit conditions and super-contact conditions, we obtain the explicit solution of the HJB equation and demonstrate the verification theorem to ensure that the solution obtained is indeed optimal. 
Our optimal feedback consumption can be categorized into five parts: (1) adjusting consumption valley, (2) maintaining consumption at the valley level (flat), (3) falling between the consumption peak and valley, (4) maintaining consumption at the peak level (flat), and (5) adjusting consumption peak.
The subsequent numerical results also validate our model. 
As discussed earlier, we have achieved both a relatively smooth consumption path and a substantial consumption utility. 
We avoid the losses incurred by excessive consumption in the Merton problem while also steering clear of the utility loss associated with excessive focus on psychological states. 

The remainder of this paper is organized as follows. 
Section \ref{sec: formulate} introduces the market model with multiple reference levels for the consumption plan and formulates the stochastic control problem. 
Section \ref{sec: solu} heuristically derives the solution of the HJB variational inequality and presents the main results on the optimal feedback consumption and investment controls. 
Section \ref{sec: property} discusses some quantitative properties of the optimal controls. 
Section \ref{sec: numerical} presents some numerical examples to illustrate the sensitivity analyses, the simulation results, and their financial implications. 
The proofs are collected in Appendix \ref{sec:Appd}.

\section{Market Model and Problem Formulation}\label{sec: formulate}

Let $(\Omega, \mathcal{F}, \mathbb{F}, \mathbb{P})$ be a filtered probability space with $\mathbb{F}=(\mathcal{F}_t)_{t\geq 0}$ satisfying the usual conditions.
The financial market consists of one riskless asset and one risky asset. The riskless asset price follows $dB_t=rB_tdt$,
where $r>0$ denotes the interest rate.
The price of the risky asset is governed by the geometric Brownian motion
\begin{equation*}
dS_t = \mu S_tdt + \sigma S_tdW_t,\quad t\geq0
\end{equation*}
where $\mu>0$ is the expected return rate, $\sigma>0$ is the volatility, and $(W_{t})_{t\geq 0}$ is a standard $\mathbb{F}$-adapted Brownian motion.
The Sharpe ratio of the risky asset is denoted by $\kappa = \frac{\mu-r}{\sigma}$.

Let $(\pi_{t})_{t\geq 0}$ be the amount of wealth that the agent allocates in the risky asset and let $(c_t)_{t\geq0}$ be the consumption rate of the agent.
The self-financing wealth process $(X_t)_{t\geq 0}$ satisfies
\begin{equation}\label{eq: wealth_system}
dX_t= \big(rX_t + (\mu-r)\pi_t - c_t \big)dt + \sigma\pi_tdW_t, \quad t\geq 0
\end{equation}
with the initial wealth $X_0=x\geq 0$.

Let us denote the historical consumption running maximum and minimum respectively by
\begin{align*}
&H_{1,t} = h_1\vee ( \sup\limits_{s\leq t} c_s),\quad t\geq 0,\\
&H_{2,t} = h_2\wedge(\inf\limits_{s\leq t} c_s),\quad t\geq 0
\end{align*}
with $h_1>h_2>0$. We consider the path-dependent impacts by past consumption pattern in the sense that some adjustment costs incur when the historical consumption maximum/ minimum is updated upwards/downwards. Therefore, we are interested in the following preference to measure the consumption performance that 
$$
\E\bigg[\int_0^\infty e^{-\delta t} U(c_t)dt - \alpha\int_0^\infty e^{-\delta t}dV_{1t}^+ - \beta \int_0^\infty e^{-\delta t}dV_{2t}^-\bigg].
$$
Here $\delta>0$ is the subjective discount rate, $\alpha$ and $\beta$ stand for the adjustment cost parameters of the historical consumption maximum and minimum. $U(\cdot)$ and $V(\cdot)$ are utility functions and
$V_{1,t}^+$ is the cumulative increment of the process $V(H_{1,t})$ under the utility $V(\cdot)$, $V_{2,t}^-$ is the cumulative decrement of the process $V(H_{2,t})$ under the utility $V(\cdot)$, i.e.,
$$
\begin{aligned}
V_{1t}^+ = V(H_{1,t}) - V(h_1),\quad
V_{2t}^- = V(h_2) - V(H_{2,t}).
\end{aligned}
$$
That is, the adjustment costs from the changes in the utilities of historical consumption maximum and minimum are included in the objective function of the agent to determine the current consumption rate.

Similar to the discussions in \cite{ChoiJeonKoo2022JET}, the agent will not increase the past spending maximum at $t$ unless $\frac{1}{\delta} - \beta > 0$, as the total utility gain would not be positive otherwise. Therefore, we impose following assumption throughout the paper.
\begin{assumption}\label{assume: SA}
$
0 \leq \alpha \leq \frac{1}{\delta}, ~ \beta \geq 0
$.
\end{assumption}
%The case that $\beta < 0$ is interpreted as utility gains from increasing consumption peak.
%(Literature, linked with this type of loss aversion)

The control pair $(c, \pi)$ is said to be \textit{admissible} if $c$ is $\mathbb{F}$-predictable and non-negative, $\pi$ is $\mathbb{F}$-progressively measurable and non-negative, both satisfy the integrability condition $\int_0^{T} (c_t+\pi_t^\top \pi_t)dt<\infty$ a.s. for any $T>0$, as well as the no bankruptcy condition holds that $X_t\geq 0$ a.s. for $t\geq 0$. We use $\mathcal{A}(x,h_1,h_2)$ to denote the set of admissible controls $(c, \pi)$.

For the given initial wealth $x>0$, the initial past consumption maximum $h_1>0$, and the initial past consumption minimum $h_2\in(0, h_1)$, we consider the stochastic control problem
\begin{equation}\label{eq: preference}
u(x, h_1, h_2) = \sup\limits_{(c_t,\pi_t) \in \mathcal{A}(x,h_1,h_2)}
\E\bigg[\int_0^\infty e^{-\delta t}\big(U(c_t)dt -\alpha dV_{1t}^+ - \beta dV_{2t}^- \big) \bigg].
\end{equation}

It is assumed in the present paper that both $U(\cdot)$ and $V(\cdot)$ are the same power utility 
\begin{equation}\label{eq: utility_felicity}
U(x) = V(x) = \frac{x^{1-\gamma}}{1-\gamma}, ~~\gamma>0,~ \gamma \neq 1.
\end{equation}

We also impose the following assumption on model parameters:
\begin{assumption}\label{assume: exist}
$
K_0 = r + \frac{\delta-r}{\gamma} + \frac{\gamma-1}{2\gamma^2}\kappa^2 > 0
$.
\end{assumption}

\section{Heuristic Derivation of the Solution}\label{sec: solu}
In this section, we analyze the HJB equation associated with the problem \eqref{eq: preference} and characterize the optimal feedback control $c^*(x,h)$ and $\pi^*(x,h)$. If one does not adjust the historical consumption peak/valley and if $u(x,\cdot)$ is $C^2$ in $x$, the first order condition gives the optimal portfolio in a feedback form by $\pi^{\ast}(x,h)=-\frac{\mu-r}{\sigma^2}\frac{u_x}{u_{xx}}$. We then arrive at the HJB equation
\begin{equation} \label{eq: HJB_main}
\sup_{c\in [h_2,h_1]} \left[ U(c)- c u_x \right] - \delta u + rx u_x -\frac{\kappa^2}{2}\frac{u_x^2}{u_{xx}} = 0,\ \ \ \ \forall x\geq0,h_1\geq h_2 \geq 0.
\end{equation}

We plan to characterize some thresholds (depending on $h_1$ and $h_2$) for the wealth level $x$ and thus the auxiliary value function, the optimal portfolio and consumption can be expressed in each region. Similar to Deng et al. \cite{DengLiPY2022FS}, we can heuristically decompose the domain into several regions based on the first order condition of $c$ and express the HJB equation \eqref{eq: HJB_main} piecewisely.
When no adjustment is optimal, wealth $x$ lies between two thresholds, i.e., $x\in[x_\mathrm{gloom}(h_1,h_2), x_\mathrm{lavs}(h_1,h_2)]$ and thus $u_x \in [\uy(h_1,h_2), \by(h_1,h_2)]$ for some thresholds $x_\mathrm{gloom}(h_1,h_2), x_\mathrm{lavs}(h_1,h_2)$, and $\uy(h_1,h_2), \by(h_1,h_2)$.
Then HJB equation \eqref{eq: HJB_main} can be rewritten as:
\begin{equation}\label{eq: HJB_region_separate}
-\delta u + rx u_x -\frac{\kappa^2}{2}\frac{u_x^2}{u_{xx}} + \tU(u_x, h_1, h_2) = 0,
\end{equation}
where $\tU(q, h_1, h_2)$ is defined by:
$$
\tU(q, h_1, h_2) := \sup\limits_{c\in[h_2,h_1]}\big(U(c) - cq\big) = \left\{\begin{aligned}
&\frac{h_1^{1-\gamma}}{1-\gamma} - h_1q, & \mbox{if } %\uy(h_1,h_2) \leq
q < h_1^{-\gamma},\\
&-\frac{1}{\gamma^*}q^{\gamma^*}, &\mbox{if } h_1^{-\gamma} \leq q \leq h_2^{-\gamma},\\
&\frac{h_2^{1-\gamma}}{1-\gamma} - h_2q, &\mbox{if } h_2^{-\gamma} < q ,%\leq \by(h_1,h_2),
\end{aligned}\right.
$$
where $\gamma^* = -\frac{1-\gamma}{\gamma}$.

According to the problem, we derive some boundary conditions which can be used to solve HJB equation \eqref{eq: HJB_region_separate}.
First, the value function $u$ admits the homogeneous property that
\begin{equation}\label{eq: value_homo}
u(x,h_1,h_2) = h_1^{1-\gamma}u(x/h_1, 1, h_2/h_1) = h_2^{1-\gamma}u(x/h_2, h_1/h_2, 1).
\end{equation}
In addition, as $h_1\rightarrow+\infty$ and $h_2\rightarrow0$, our problem degenerates to the classical Merton's problem and it holds that
\begin{equation}\label{eq: condition_merton}
\lim\limits_{h_1\rightarrow+\infty, h_2\rightarrow0} u(x, h_1, h_2) = c_{_M} x^{1-\gamma},
\end{equation}
where $c_{_M}$ is a constant corresponding to classical Merton's problem.
Suppose the agents adjust past spending maximum $H_{1,t}$ over an infinitesimal time period $[t, t + dt)$, then the benefit can be computed in utility terms as follows:
\begin{equation}\label{eq: benefit_adjust_h1}
u(x_{t-}, H_{1,t-} + dh_1, H_{2, t-}) - u(x_{t-}, H_{1,t-}, H_{2,t-}) \approx u_{h_1}(x_{t-}, H_{1,t-}, H_{2,t-})dh_1.
\end{equation}
Suppose that the agents adjust past spending minimum $H_{2,t}$ over an infinitesimal time period $[t, t + dt)$, then the benefit can be computed in utility terms as follows:
\begin{equation}\label{eq: benefit_adjust_h2}
\begin{aligned}
u(x_{t-}, H_{1,t-} , H_{2, t-} + dh_2) - u(x_{t-}, H_{1,t-}, H_{2,t-}) \approx u_{h_2}(x_t, H_{1,t-}, H_{2,t-})dh_2,
\end{aligned}
\end{equation}
In addition, the adjustment cost can be represented as:
$$
\mathrm{Adjustment~Cost} =
\left\{\begin{aligned}
&\alpha dV_{1t}^+ = \alpha V'(H_{1,t-})dh_1, &\mbox{ if } dh_1 > 0,\\
&\beta dV_{2t}^- = -\beta V'(H_{2,t-})dh_2, &\mbox{ if } dh_2 < 0.
\end{aligned}\right.
$$
Thus, the adjustment is optimal only when the benefit is greater than or equal to the cost, i.e., $$\left\{\begin{aligned}
&u_{h_1}(x_{t-}, H_{1,t-}, H_{2,t-}) \geq \alpha V'(H_{1,t-}), & \mbox{if } dh_1 > 0,\\
&-u_{h_2}(x_{t-}, H_{1,t-}, H_{2,t-}) \geq \beta V'(H_{2,t-}), & \mbox{if } dh_2 < 0.
\end{aligned}\right.$$
Note that $u_{h_1}(x_{t-}, H_{1,t-}, H_{2,t-})$ and $u_{h_2}(x_{t-}, H_{1,t-}, H_{2,t-})$, the respective marginal valuations of consumption and historical consumption peak implied by the value function, are different from the marginal utility of consumption peak $V'(H_{1,t-})$ or valley $V'(H_{2,t-})$, as indicated in equation \eqref{eq: benefit_adjust_h1} and \eqref{eq: benefit_adjust_h2}.
$u_{h_1}(x_{t-}, H_{1,t-}, H_{2, t-})$ and $u_{h_2}(x_{t-}, H_{1, t-}, H_{2,t-})$ measure the benefits of adjusting historical consumption peak/valley, considering the total effects of the decision on future utility values.
$V'(H_{1,t-})$ and $V'(H_{2,t-})$ measure the benefit of marginal change in habit over an infinitesimal time period $[t, t+dt)$, and is proportional to the cost of adjustment.
Therefore, no adjustment is optimal, i.e., the agent does not adjust past spending maximum if $u_{h_1}(x_{t-}, H_{1,t-}, H_{2,t-}) < \alpha V'(H_{1,t-})$, and does not adjust past spending minimum if $u_{h_2}(y_{t-}, H_{1,t-}, H_{2,t-}) > -\beta V'(H_{2,t-})$.

In addition, the smooth pasting condition implies that we have
\begin{equation}\label{eq: condition_adjust_h1}
u_{h_1}(x, h_1, h_2) = \alpha V'(h_1),
\end{equation}
when the agent increases past spending maximum, and
\begin{equation}\label{eq: condition_adjust_h2}
u_{h_2}(x, h_1, h_2) = -\beta V'(h_2),
\end{equation}
when the agent decreases the past spending minimum.
In summary, the value function $u$ satisfies the following HJB variational inequality
$$
\max\bigg\{-\delta u + rx u_x -\frac{\kappa^2}{2}\frac{u_x^2}{u_{xx}} + \tU(u_x, h_1, h_2),
u_{h_1}(x, h_1, h_2) - \alpha V'(h_1), -u_{h_2}(x, h_1, h_2) - \beta V'(h_2)\bigg\} = 0.
$$
%Hence, the agent's optimal decision can be illustrated in three regions in the state space: non-adjustment region \textbf{NR}, increasing~region \textbf{IR}, and decreasing region \textbf{DR} as follows:
%$$
%\begin{aligned}
%\textbf{NR}&=\{(x,h_1,h_2)\in\mathbb{R}_+^3: \uy(h_1,h_2)\leq u_x \leq \by(h_1,h_2)\}, \\
%\textbf{IR}&=\{(x,h_1,h_2)\in\mathbb{R}_+^3: u_x < \uy(h_1,h_2)\}, \\
%\textbf{DR}&=\{(x,h_1,h_2)\in\mathbb{R}_+^3: u_x > \by(h_1,h_2)\}.
%\end{aligned}
%$$
%In domain \textbf{NR}, the optimal consumption lies between $[h_1, h_2]$ and does not adjustment past spending maximum/minimum; in domain \textbf{IR}, the agent increases the consumption peak; in domain \textbf{DR}, the agent decreases the consumption valley.
%Note that the \textbf{NR} region can also be separated to three regions:
%$$
%\begin{aligned}
%\textbf{NR}_{1} &= \{(x,h_1,h_2)\in\mathbb{R}_+^3: \uy(h_1,h_2)\leq u_x < h_1^{-\gamma}\}, \\
%\textbf{NR}_{2} &= \{(x,h_1,h_2)\in\mathbb{R}_+^3: h_1^{-\gamma} \leq u_x \leq h_2^{-\gamma}\}, \\ \textbf{NR}_{3} &= \{(x,h_1,h_2)\in\mathbb{R}_+^3: h_2^{-\gamma}\leq u_x \leq \by(h_1,h_2)\}.
%\end{aligned}
%$$
%In region \textbf{NR}$_{1}$, the optimal consumption equals but not updates the past spending maximum $h_1$; in region \textbf{NR}$_{2}$, the optimal consumption lies in $(h_1, h_2)$; in region \textbf{NR}$_{3}$, the optimal consumption equals but not updates the past spending minimum $h_2$.

We then employ the dual transform only with respect to variable $x$ and treat $h_1$ and $h_2$ as parameters that
$$
v(y,h_1,h_2) := \sup\limits_{x\geq0}(u(x,h_1,h_2) - xy), ~\mathrm{for}~ h_1\geq h_2 \geq 0.
$$
For any given $(x,h_1,h_2)$, we define $y(x,h_1,h_2) := u_x(x,h_1,h_2)$ (short as $y$).
Similar to \cite{DengLiPY2022FS}, using the dual transform,  we can obtain the dual PDE
\begin{equation}\label{eq: HJB_dual}
\begin{aligned}
\delta v - \frac{\kappa^2 y^2}{2}v_{yy} - (\delta-r)yv_y
=
\tU(y, h_1, h_2).
\end{aligned}
\end{equation}
By virtue of the duality relationship and \eqref{eq: condition_adjust_h1} and \eqref{eq: condition_adjust_h2}, we have the following boundary conditions for the dual PDE
\begin{equation}\label{eq: dual_boundary_adjust1}
\begin{aligned}
v_{h_1}(y, h_1, h_2) = \alpha V'(h_1), &\mbox{~for~} y \leq \uy(h_1,h_2), \\
v_{h_2}(y, h_1, h_2) = -\beta V'(h_2), &\mbox{~for~} y \geq \by(h_1,h_2).
\end{aligned}
\end{equation}
Therefore, the dual value function $v$ satisfies the following HJB variational inequality

\begin{equation}\label{eq: HJB_dual_ineq}
\max\bigg\{-\delta v + \frac{\kappa^2 y^2}{2}v_{yy} + (\delta-r)yv_y + \tU(y, h_1, h_2),
v_{h_1}(y, h_1, h_2) - \alpha V'(h_1), -v_{h_2}(y, h_1, h_2) - \beta V'(h_2)\bigg\} = 0.
\end{equation}

The super-contact conditions are written by
\begin{equation}\label{eq: dual_boundary_adjust2}
\begin{aligned}
&v_{yh_1}(y, h_1, h_2) = 0,  &\mbox{for~} y \leq \uy(h_1,h_2), \\
&v_{yh_2}(y, h_1, h_2) = 0,  &\mbox{~for~} y \geq \by(h_1,h_2).
\end{aligned}
\end{equation}
The homogeneous property \eqref{eq: value_homo} shows that
\begin{equation}\label{eq: dual_homo}
v(y,h_1,h_2) = h_1^{1-\gamma}v(yh_1^{\gamma}, 1, h_2/h_1) = h_2^{1-\gamma}v(yh_2^{\gamma}, h_1/h_2,1),
\end{equation}
and condition \eqref{eq: condition_merton} implies
\begin{equation}\label{eq: dual_merton}
\lim\limits_{h_1\rightarrow+\infty, h_2\rightarrow0} v(y, h_1, h_2)
= c_{_d} y^{\gamma^*}.
\end{equation}

\begin{proposition}\label{prop: solu_dual_value}
Under 
%Assumption \ref{assume: exist} and Assumption \ref{assume: SA}, 
the boundary conditions \eqref{eq: dual_boundary_adjust1}-\eqref{eq: dual_boundary_adjust2}, homogeneous property \eqref{eq: dual_homo}, and smooth-fit conditions at points $y = h_1^{-\gamma}$ and $y = h_2^{-\gamma}$, the dual value function $v(y,h_1,h_2)$ is given by
\begin{equation}\label{eq: solu_dual_value}
v(y,h_1,h_2) =
\left\{\begin{aligned}
&C_1(h_1,h_2)y^{m_1} + C_2(h_1)y^{m_2} + \frac{h_1^{1-\gamma}}{\delta(1-\gamma)} - \frac{yh_1}{r},  &\mbox{ if } z_\alpha h_1^{-\gamma} \leq y < h_1^{-\gamma}, \\
&C_3(h_2)y^{m_1} + C_4(h_1)y^{m_2} +
\frac{2}{\kappa^2 \gamma^*(\gamma^*-m_1)(\gamma^*-m_2)}y^{\gamma^*},  &\mbox{ if } h_1^{-\gamma}\leq y \leq h_2^{-\gamma}, \\
&C_5(h_2)y^{m_1} + C_6(h_1,h_2)y^{m_2} + \frac{h_2^{1-\gamma}}{\delta(1-\gamma)} - \frac{yh_2}{r},  &\mbox{ if }  h_2^{-\gamma} < y \leq z_\beta h_2^{-\gamma}, \\
&v\big(y, \bh_1(y), h_2 \big)
+ \frac{\alpha}{1-\gamma}\big(h_1^{1-\gamma} - \bh_1(y)^{1-\gamma}\big), &\mbox{ if } y < z_\alpha h_1^{-\gamma}, \\
&v\big(y, h_1, \uh_2(y) \big)
-\frac{\beta}{1-\gamma}\big(h_2^{1-\gamma} - \uh_2(y)^{1-\gamma}\big), &\mbox{ if } y > z_\beta h_2^{-\gamma},
\end{aligned}\right.
\end{equation}
where $\bh_1(y):= (y/z_\alpha)^{-1/\gamma}$, $\uh_2(y):= (y/z_\beta)^{-1/\gamma}$, $m_1\in (1,+\infty)$, $m_2\in(-\infty, \min(\gamma^*, 0))$ are the two roots of the  quadratic equation $\frac{\kappa^2}{2}m^2 + \bigg(\delta-r-\frac{\kappa^2}{2}\bigg) m - \delta = 0$, and $C_1(h_1,h_2)$, $C_2(h_1)$, $C_3(h_2)$, $C_4(h_1)$, $C_5(h_2)$, and $C_6(h_1,h_2)$ are functions of $h_1$ and $h_2$ given by
\begin{equation}\label{eq: C_solu}
\begin{aligned}
C_1(h_1,h_2) =& C_3(h_2) + \frac{2(1-\gamma^*)}{\kappa^2(m_1-m_2)m_1(m_1-1)(m_1-\gamma^*)}h_1^{1 + \gamma(m_1-1)}, \\
C_2(h_1) =& \frac{1-\gamma^*}{(m_1-m_2)(m_2-\gamma^*)}\bigg(\frac{m_1(\alpha\delta-1)}{\delta}z_\alpha^{-m_2} + \frac{m_1-1}{r}z_\alpha^{1-m_2}\bigg) h_1^{1+\gamma(m_2-1)}, \\
C_3(h_2) =& C_5(h_2) + \frac{2(\gamma^*-1)}{\kappa^2(m_1-m_2)m_1(m_1-1)(m_1-\gamma^*)}h_2^{1 + \gamma(m_1-1)}, \\
C_4(h_1) =& C_2(h_1) - \frac{2(\gamma^*-1)}{\kappa^2(m_1-m_2)m_2(m_2-1)(m_2-\gamma^*)}h_1^{1 + \gamma(m_2-1)}, \\
C_5(h_2) =& \frac{1-\gamma^*}{(m_1-m_2)(m_1-\gamma^*)}\bigg(\frac{m_2(\beta\delta+1)}{\delta}z_\beta^{-m_1} + \frac{1-m_2}{r}z_\beta^{1-m_1}\bigg)h_2^{1 + \gamma(m_1-1)}, \\
C_6(h_1,h_2) =& C_4(h_1) - \frac{2(1-\gamma^*)}{\kappa^2(m_1-m_2)m_2(m_2-1)(m_2-\gamma^*)}h_2^{1 + \gamma(m_2-1)}. \\
\end{aligned}
\end{equation}
Here, $z_\alpha \in (0, 1-\alpha\delta]$ and $z_\beta \in [1+\beta\delta, +\infty)$ are solutions to the following algebraic equations
\begin{equation}\label{eq: z_alpha_beta}
\begin{aligned}
&\frac{2}{\kappa^2 m_1(m_1-1)}z_\alpha^{m_1} + \frac{z_\alpha}{r}(m_2-1) + m_2\bigg(\alpha - \frac1\delta\bigg) = 0, \\
&\frac{2}{\kappa^2 m_2(m_2-1)}z_\beta^{m_2} + \frac{z_\beta}{r}(m_1-1) - m_1\bigg(\beta + \frac1\delta\bigg) = 0.\end{aligned}
\end{equation}
%
%The regions $\mathbf{NR}_{d,1}$, $\mathbf{NR}_{d,2}$, $\mathbf{NR}_{d,3}$, $\mathbf{IR}_d$, $\mathbf{DR}_d$ are rewritten as:
%%
%\begin{equation}\label{eq: dual_regions_solu}
%\begin{aligned}
%\mathbf{NR}_{d,1} &= \{(y,h_1,h_2) \in \mathbb{R}_+^3: h_1\geq h_2,~ z_\alpha h_1^{-\gamma} \leq y < h_1^{-\gamma} \},\\
%\mathbf{NR}_{d,2} &= \{(y,h_1,h_2) \in \mathbb{R}_+^3: h_1\geq h_2,~ h_1^{-\gamma} \leq y \leq h_2^{-\gamma} \},\\
%\mathbf{NR}_{d,3} &= \{(y,h_1,h_2) \in \mathbb{R}_+^3: h_1\geq h_2,~ h_2^{-\gamma} < y \leq z_\beta h_2^{-\gamma} \},\\
%\mathbf{IR}_d &= \{(y,h_1,h_2) \in \mathbb{R}_+^3: h_1\geq h_2,~ y < z_\alpha h_1^{-\gamma} \}, \\
%\mathbf{DR}_d &= \{(y,h_1,h_2) \in \mathbb{R}_+^3: h_1\geq h_2,~ y > z_\beta h_2^{-\gamma} \}.\\
%\end{aligned}
%\end{equation}
%%
\end{proposition}
The following theorem states the main results of this paper, provides the optimal investment, consumption strategies in piecewise feedback form using the variables $y, h_1,$ and $h_2$.
\begin{theorem}[Verification Theorem]\label{thm: verification}
Let $(x,h_1,h_2)\in\mathbb{R}_+^3$, $h_1 \geq h_2$, where $x\geq0$ stands for the initial wealth, $h_1$ and $h_2$ are the initial reference levels.
For $(y, h_1, h_2) \in \mathbb{R}_+^3$ and $h_1\geq h_2$, let us define the feedback functions:
\begin{equation}\label{eq: consume_feedback}
c^\dag(y, h_1, h_2) =
\left\{\begin{aligned}
&h_1, &\mbox{if }& z_\alpha h_1^{-\gamma} \leq y < h_1^{-\gamma}, \\
&y^{-\frac{1}{\gamma}}, &\mbox{if }& h_1^{-\gamma} \leq y \leq h_2^{-\gamma}, \\
&h_2, &\mbox{if }& h_2^{-\gamma} < y \leq z_\beta h_2^{-\gamma}, \\
&\bh_1(y), &\mbox{if }& y < z_\alpha h_1^{-\gamma}, \\
&\uh_2(y), &\mbox{if }& y > z_\beta h_2^{-\gamma}, \\
\end{aligned}\right.
\end{equation}
\begin{equation}\label{eq: portfolio_feedback}
\begin{aligned}
&\pi^\dag(y, h_1, h_2)
=
\frac{\mu-r}{\sigma^2} \cdot\\
&\left\{\begin{aligned}
&m_1(m_1-1)C_1(h_1, h_2)y^{m_1-1} + m_2(m_2-1)C_2(h_1)y^{m_2-1}, &\mbox{if }& z_\alpha h_1^{-\gamma} \leq y < h_1^{-\gamma}, \\
&m_1(m_1-1)C_3(h_2)y^{m_1-1} + m_2(m_2-1)C_4(h_1)y^{m_2-1} \\
&+ \frac{2(\gamma^*-1)}{\kappa^2(\gamma^*-m_1)(\gamma^*-m_2)}y^{\gamma^*-1}, &\mbox{if }& h_1^{-\gamma} \leq y \leq h_2^{-\gamma}, \\
&m_1(m_1-1)C_5(h_2)y^{m_1-1} + m_2(m_2-1)C_6(h_1,h_2)y^{m_2-1}, &\mbox{if }& h_2^{-\gamma} < y \leq z_\beta h_2^{-\gamma}, \\
&m_1(m_1-1)C_1\big(\bh_1(y), h_2\big)y^{m_1-1} + m_2(m_2-1)C_2\big(\bh_1(y)\big)y^{m_2-1}, &\mbox{if }& y < z_\alpha h_1^{-\gamma}, \\
&m_1(m_1-1)C_5\big( \uh_2(y)\big)y^{m_1-1} + m_2(m_2-1)C_6\big(h_1,\uh_2(y)\big)y^{m_2-1}, &\mbox{if }& y > z_\beta h_2^{-\gamma}. \\
\end{aligned}\right.
\end{aligned}
\end{equation}
We consider the process $Y_t(y) := ye^{\delta t}\xi_t$, where $\xi_t := e^{-\big(r+\frac{\kappa^2}{2} \big)t-\kappa W_t}$ is the discounted rate state price density process, and $y^* = y^*(x, h_1, h_2)$ is the unique solution to the budget constraint
$$
\E\bigg[\int_0^\infty c^\dag(Y_s(y), H^\dag_{1,s}(y), H^\dag_{2,s}(y)) \xi_t\bigg] = x,
$$
where
$$
\begin{aligned}
H^\dag_{1,t}(y) &:= h_1\vee\sup\limits_{s\leq t}c^\dag(Y_s(y), H_{1,s}^\dag(y)) = h_1\vee\big(\inf\limits_{s\leq t}Y_s(y)/z_\alpha\big)^{-\frac1\gamma}, \\
H^\dag_{2,t}(y) &:= h_2\wedge\inf\limits_{s\leq t}c^\dag(Y_s(y), H_{2,s}^\dag(y)) = h_2\wedge\big(\sup\limits_{s\leq t}Y_s(y)/z_\beta\big)^{-\frac1\gamma}, \\
\end{aligned}
$$
are the optimal references processes corresponding to any fixed $y>0$.
The value function $u(x,h_1,h_2)$ can be attained by employing the optimal consumption and portfolio strategies in the feedback form that $c_t^* = c^\dag(Y_t^*, H_{1,t}^*, H_{2,t}^*)$ and $\pi_t^* = \pi^\dag(Y_t^*, H_{1,t}^*, H_{2,t}^*)$ where $Y_t^* = Y_t(y^*)$, $H_{1,t}^* = H_{1,t}^\dag(y^*)$, and $H_{2,t}^* = H_{2,t}^\dag(y^*)$.

The process $H_{1,t}^*$ is strictly increasing if and only if $Y_t^* = z_\alpha {H_{1,t}^*}^{-\gamma}$.
If we have $y^*(x,h_1,h_2) < z_\alpha h_1^{-\gamma}$ at the initial time, the optimal consumption creates a new peak and brings $H_{1, 0-}^* = h_1$ jumping immediately to a higher level $H_{1,0}^* = (y^*(x, h_1, h_2)/z_\alpha)^{-\frac{1}{\gamma}}$ such that $t = 0$ becomes the only jump time of $H_{1,t}^*$.
Similarly, the process $H_{2,t}^*$ is strictly decreasing in and only if $Y_t^* = z_\beta {H_{2,t}^*}^{-\gamma}$.
If we have $y^*(x,h_1,h_2) > z_\beta h_2^{-\gamma}$ at the initial time, the optimal consumption creates a new valley and brings $H_{2, 0-}^* = h_2$ jumping immediately to a lower level $H_{2,0}^* = (y^*(x, h_1, h_2)/z_\beta)^{-\frac{1}{\gamma}}$ such that $t = 0$ becomes the only jump time of $H_{2,t}^*$.
\end{theorem}

Using the dual relationship between $u$ and $v$, we have the optimal choice $x$ satisfying $u_x(x,h) = y$ admits the expression $x = g(\cdot, h_1, h_2) = -v_y(\cdot,h_1, h_2)$. Define $f(\cdot, h_1, h_2)$ as the inverse of $g(\cdot, h_1, h_2)$, then $u(x,h_1,h_2) = v(f(x,h_1,h_2),h_1,h_2) + xf(x,h_1,h_2)$.
Note that $f$ should have piecewise form across different regions.
By the definition of $g$, the invertibility of the map $x\mapsto g(x,h_1,h_2)$ is guaranteed by the following lemma.
\begin{lemma}\label{lemma: v_convex}
The dual value function $v(y, h_1, h_2)$ in \eqref{eq: solu_dual_value} is strictly convex in $y$, i.e., $v_{yy}(y, h_1, h_2) > 0$,  so that the inverse Legendre transform $u(x, h_1, h_2) = \inf\limits_{y>0}\big(v(y,h_1,h_2) + xy\big)$ is well defined.
Moreover, it implies that the feedback optimal portfolio $\pi^*(y, h_1,h_2) > 0$ all the time.
\end{lemma}

Thanks to Lemma \ref{lemma: v_convex}, we can apply the inverse Legendre transform to the solution $v(y, h_1, h_2)$, and thus characterize the following four boundary curves $x_\mathrm{lavs}(h_1, h_2)$, $x_\mathrm{peak}(h_1, h_2)$, $x_\mathrm{valy}(h_1, h_2)$, $x_\mathrm{gloom}(h_1, h_2)$ that
\begin{equation}\label{eq: x_boundaries}
\begin{aligned}
x_\mathrm{lavs}(h_1, h_2) &:= -m_1C_1(h_1, h_2)(z_\alpha h_1^{-\gamma})^{m_1-1} - m_2C_2(h_1)(z_\alpha h_1^{-\gamma})^{m_2-1} + \frac{h_1}{r},\\
x_\mathrm{peak}(h_1, h_2) &:= -m_1C_1(h_1, h_2)(h_1^{-\gamma})^{m_1-1} - m_2C_2(h_1)(h_1^{-\gamma})^{m_2-1} + \frac{h_1}{r}, \\
x_\mathrm{valy}(h_1, h_2) &:= -m_1C_5(h_2)(h_2^{-\gamma})^{m_1-1} - m_2C_6(h_1, h_2)(h_2^{-\gamma})^{m_2-1} + \frac{h_2}{r}, \\
x_\mathrm{gloom}(h_1, h_2) &:= -m_1C_5(h_2)(z_\beta h_2^{-\gamma})^{m_1-1} - m_2C_6(h_1, h_2)(z_\beta h_2^{-\gamma})^{m_2-1} + \frac{h_2}{r},
\end{aligned}
\end{equation}
and the optimal feedback function of optimal consumption satisfies: (i) $c^*(x, h_1, h_2) > h_1$ if $x >x_\mathrm{lavs}(h_1, h_2)$; (ii) $c^*(x, h_1, h_2) = h_1$ if $x\in [x_\mathrm{peak}(h_1, h_2), x_\mathrm{lavs}(h_1, h_2)]$; (iii) $c^*(x, h_1, h_2) \in (h_2, h_1)$ if $x\in (x_\mathrm{valy}(h_1, h_2), x_\mathrm{peak}(h_1,h_2))$; (iv) $c^*(x, h_1, h_2) = h_2$ if $x\in [x_\mathrm{gloom}(h_1, h_2), x_\mathrm{valy}(h_1, h_2)]$, and (v) $c^*(x, h_1, h_2) < h_2$ if $x < x_\mathrm{gloom}(h_1, h_2)$.\\
In addition, we define $f(x, h_1, h_2)$ to be the respective solutions to the five equations that
%
%\begin{equation}\label{eq: f_obtain}
%\begin{aligned}
%x =& -m_1C_1\big(\bh_1\big(f(x, h_1, h_2)\big), h_2\big)f(x, h_1, h_2)^{m_1-1} \\
%&- m_2C_2(\bh_1\big(f(x, h_1, h_2)\big)\big)f(x, h_1, h_2)^{m_2-1}+\frac{\bh_1\big(f(x, h_1, h_2)\big)}{r}, &\mbox{if }& x > x_\mathrm{lavs}(h_1, h_2), \\
%x =& -m_1C_1(h_1, h_2)f(x, h_1, h_2)^{m_1-1} - m_2C_2(h_1)f(x, h_1, h_2)^{m_2-1}+\frac{h_1}{r}, &\mbox{if }& x_\mathrm{peak}(h_1,h_2) < x \leq x_\mathrm{lavs}(h_1, h_2), \\
%x =& -m_1C_3( h_2)f(x, h_1, h_2)^{m_1-1} - m_2C_4(h_1)f(x, h_1, h_2)^{m_2-1} \\
%&- \frac{2}{\kappa^2 (\gamma^*-m_1)(\gamma^*-m_2)}f(x, h_1, h_2)^{\gamma^*-1}, &\mbox{if }& x_\mathrm{valy}(h_1,h_2) \leq x \leq x_\mathrm{peak}(h_1,h_2), \\
%x = &-m_1C_5( h_2)f(x, h_1, h_2)^{m_1-1} - m_2C_6(h_1, h_2)f(x, h_1, h_2)^{m_2-1}+\frac{h_2}{r}, &\mbox{if }& x_\mathrm{gloom}(h_1,h_2) \leq x < x_\mathrm{valy}(h_1, h_2), \\
%x = &-m_1C_5\big( \uh_2\big(f(x, h_1, h_2)\big)\big)f(x, h_1, h_2)^{m_1-1} \\
%&- m_2C_6\big(h_1, \uh_2\big(f(x, h_1, h_2)\big)\big)f(x, h_1, h_2)^{m_2-1}+\frac{\uh_2\big(f(x, h_1, h_2)\big)}{r}, & \mbox{if }& x < x_\mathrm{gloom}(h_1, h_2).
%\end{aligned}
%\end{equation}
\begin{equation}\label{eq: f_obtain}
\begin{aligned}
x =& -m_1C_1\big(\bh_1\big(f\big), h_2\big)f^{m_1-1}- m_2C_2(\bh_1\big(f\big)\big)f^{m_2-1}+\frac{\bh_1\big(f\big)}{r}, &\mbox{if }& x > x_\mathrm{lavs}, \\
x =& -m_1C_1(h_1, h_2)f^{m_1-1} - m_2C_2(h_1)f^{m_2-1}+\frac{h_1}{r}, &\mbox{if }& x_\mathrm{peak} < x \leq x_\mathrm{lavs}, \\
x =& -m_1C_3( h_2)f^{m_1-1} - m_2C_4(h_1)f^{m_2-1} - \frac{2}{\kappa^2 (\gamma^*-m_1)(\gamma^*-m_2)}f^{\gamma^*-1}, &\mbox{if }& x_\mathrm{valy} \leq x \leq x_\mathrm{peak}, \\
x = &-m_1C_5( h_2)f^{m_1-1} - m_2C_6(h_1, h_2)f^{m_2-1}+\frac{h_2}{r}, &\mbox{if }& x_\mathrm{gloom} \leq x < x_\mathrm{valy}, \\
x = &-m_1C_5\big( \uh_2\big(f\big)\big)f^{m_1-1} - m_2C_6\big(h_1, \uh_2\big(f(x, h_1, h_2)\big)\big)f^{m_2-1}+\frac{\uh_2\big(f\big)}{r}, & \mbox{if }& x < x_\mathrm{gloom}.
\end{aligned}
\end{equation}
Based on boundaries \eqref{eq: x_boundaries} and transform \eqref{eq: f_obtain}, the following corollary shows the analytical form for the value function, optimal consumption and optimal portfolio for problem \eqref{eq: preference}, whose proof is given in Section \ref{sec: proof_verification}.

\begin{corollary}\label{cor: main_res}
For $(x, h_1 ,h_2) \in \mathbb{R}_+^3$ and $h_1\geq h_2$, the value function $u(x,h_1,h_2)$ can be expressed in a piecewise form that
\begin{equation}\label{eq: solu_value}
\begin{aligned}
&u(x, h_1, h_2) = \\
&\left\{\begin{aligned}
&C_1(h_1,h_2)f(x, h_1, h_2)^{m_1} + C_2(h_1)f(x, h_1,h_2)^{m_2} \\
&+ \frac{h_1^{1-\gamma}}{\delta(1-\gamma)} - \frac{f(x, h_1, h_2)h_1}{r} + xf(x, h_1, h_2),  &\mbox{ if }& x_\mathrm{peak}(h_1,h_2) < x \leq x_\mathrm{lavs}(h_1, h_2), \\
&C_3(h_2)f(x,  h_2)^{m_1} + C_4(h_1)f(x, h_1)^{m_2} \\
& + \frac{2}{\kappa^2 \gamma^*(\gamma^*-m_1)(\gamma^*-m_2)}f(x, h_1, h_2)^{\gamma^*} + xf(x, h_1, h_2),  &\mbox{ if }& x_\mathrm{valy}(h_1,h_2) \leq x \leq x_\mathrm{peak}(h_1,h_2), \\
&C_5(h_2)f(x, h_1, h_2)^{m_1} + C_6(h_1,h_2)f(x, h_1, h_2)^{m_2} \\
&+ \frac{h_2^{1-\gamma}}{\delta(1-\gamma)} - \frac{f(x, h_1, h_2)h_2}{r} + xf(x, h_1, h_2),  &\mbox{ if }&  x_\mathrm{gloom}(h_1,h_2) \leq x < x_\mathrm{valy}(h_1, h_2), \\
&u\big(x, \bh_1\big(f(x, h_1, h_2)\big), h_2 \big) + \frac{\alpha}{1-\gamma}\bigg(h_1^{1-\gamma} - \bh_1\big(f(x, h_1, h_2)\big)^{1-\gamma}\bigg), &\mbox{ if }& x > x_\mathrm{lavs}(h_1, h_2), \\
&u\big(x, h_1, \uh_2\big(f(x, h_1, h_2)\big)\big) - \frac{\beta}{1-\gamma}\bigg(h_2^{1-\gamma} - \uh_2\big(f(x, h_1, h_2)\big)^{1-\gamma}\bigg), &\mbox{ if }& x < x_\mathrm{gloom}(h_1, h_2).
\end{aligned}\right.
\end{aligned}
\end{equation}
In addition, the feedback optimal consumption is proposed as follows:
\begin{equation}\label{eq: solu_consume}
c^*(x, h_1, h_2) =
\left\{\begin{aligned}
&h_1,  &\mbox{ if }& x_\mathrm{peak}(h_1,h_2) < x \leq x_\mathrm{lavs}(h_1, h_2), \\
&f(x, h_1, h_2)^{-\frac1\gamma},  &\mbox{ if }& x_\mathrm{valy}(h_1,h_2) \leq x \leq x_\mathrm{peak}(h_1,h_2), \\
&h_2,  &\mbox{ if }&  x_\mathrm{gloom}(h_1,h_2) \leq x < x_\mathrm{valy}(h_1, h_2), \\
&\bh_1\big(f(x, h_1, h_2)\big), &\mbox{ if }& x > x_\mathrm{lavs}(h_1, h_2), \\
&\uh_2\big(f(x, h_1, h_2)\big), &\mbox{ if }& x < x_\mathrm{gloom}(h_1, h_2),
\end{aligned}\right.
\end{equation}
and the optimal portfolio
\begin{equation}\label{eq: solu_portfolio}
\begin{aligned}
&\pi^*(x, h_1, h_2) = \frac{\mu-r}{\sigma^2}\cdot\\
&\left\{\begin{aligned}
&m_1(m_1-1)C_1(h_1, h_2)f(x, h_1, h_2)^{m_1-1} \\
&+ m_2(m_2-1)C_2(h_1)f(x, h_1, h_2)^{m_2-1},  &\mbox{ if }& x_\mathrm{peak}(h_1,h_2) < x \leq x_\mathrm{lavs}(h_1, h_2), \\
&m_1(m_1-1)C_3( h_2)f(x, h_1, h_2)^{m_1-1} \\
&+ m_2(m_2-1)C_4(h_1)f(x, h_1, h_2)^{m_2-1} \\
&+ \frac{2(\gamma^*-1)}{\kappa^2(\gamma^*-m_1)(\gamma^*-m_2)}f(x, h_1, h_2)^{\gamma^*-1},  &\mbox{ if }& x_\mathrm{valy}(h_1,h_2) \leq x \leq x_\mathrm{peak}(h_1,h_2), \\
&m_1(m_1-1)C_5( h_2)f(x, h_1, h_2)^{m_1-1} \\
&+ m_2(m_2-1)C_6(h_1,h_2)f(x, h_1, h_2)^{m_2-1},  &\mbox{ if }&  x_\mathrm{gloom}(h_1,h_2) \leq x < x_\mathrm{valy}(h_1, h_2), \\
&m_1(m_1-1)C_1\big(\bh_1\big(f(x, h_1, h_2)\big), h_2\big)f(x, h_1, h_2)^{m_1-1} \\
&+ m_2(m_2-1)C_2\big(\bh_1\big(f(x, h_1, h_2)\big), h_2\big)f(x, h_1, h_2)^{m_2-1}, &\mbox{ if }& x > x_\mathrm{lavs}(h_1, h_2), \\
&m_1(m_1-1)C_5\big(h_1, \uh_2\big(f(x, h_1, h_2)\big)\big)f(x, h_1, h_2)^{m_1-1} \\
&+ m_2(m_2-1)C_6\big(h_1,\uh_2\big(f(x, h_1, h_2)\big)\big)f(x, h_1, h_2)^{m_2-1}, &\mbox{ if }& x < x_\mathrm{gloom}(h_1, h_2).
\end{aligned}\right.
\end{aligned}
\end{equation}
In addition, for any initial state tuple $(X_0^*, H_{1, 0}^*, H_{2, 0}^*) = (x, h_1, h_2)\in\mathbb{R}_+^3$ and $h_1\geq h_2$, the stochastic differential equation
\begin{equation}\label{eq: SDE_opt}
dX_t^* = rX_t^*dt + \pi^*(\mu-r)dt - c^*dt + \pi^*\sigma dW_t,
\end{equation}
has a unique strong solution under the optimal feedback control $(c^*, \pi^*)$.
\end{corollary}

\section{Quantitative Properties of the Optimal Control}\label{sec: property}
In this section, we discuss the theoretical properties of the value function, optimal consumption and optimal portfolio for problem \eqref{eq: preference} based on  their analytical form in Corollary \ref{cor: main_res}. The following proposition presents the monotonicity and concavity (convexity) for the boundary curves with respect to the  historical consumption peak and consumption valley levels.
\begin{proposition}\label{prop: boundary_deri}
Boundary curves $x_\mathrm{lavs}(h_1,h_2)$ and $x_\mathrm{peak}(h_1,h_2)$ are convex increasing in $h_1$ and concave increasing in $h_2$.
Boundary curves $x_\mathrm{gloom}(h_1,h_2)$ and $x_\mathrm{valy}(h_1,h_2)$ are convex increasing in $h_2$ and concave increasing in $h_1$. 
\end{proposition}

Next, we study the problem \eqref{eq: preference} in some extreme cases.
\begin{theorem}\label{thm: sensitivity}
The following properties hold:
\begin{enumerate}
\item If $\alpha = \beta = 0$, then the model degenerates to the Merton model with $z_\alpha = z_\beta = 1$, that is, the peak and lavish points coincide, and the valley and gloom points coincide.
    As $\alpha \rightarrow \frac1\delta$, $z_\alpha\rightarrow0$, the agent does not adjust the consumption valley; as $\beta \rightarrow +\infty$, $z_\beta\rightarrow+\infty$, the agent does not adjust the consumption peak.
\item As the risk aversion parameter $\gamma$ tends towards infinity, the Gloom ratio and Lavish ratio both approximate to $1/r$, namely, for $h_1\geq h_2\geq 0$,
\begin{align*}
\lim_{\gamma\to\infty}\frac{x_\mathrm{lavs}(h_1, h_2)}{c^*(x_\mathrm{lavs}(h_1, h_2),h_1,h_2)}=\lim_{\gamma\to\infty}\frac{x_\mathrm{gloom}(h_1, h_2)}{c^*(x_\mathrm{gloom}(h_1, h_2),h_1,h_2)}=\frac{1}{r}.
\end{align*}
%\item Gloom ratio and Lavish ratio approx to $1/r$? under some assumptions
\end{enumerate}
\end{theorem}
As shown in Theorem \ref{thm: sensitivity}-(i), when both cost parameters equal zero (i.e., $\alpha=\beta=0$), our general formulation \eqref{eq: preference} degenerates to the Merton model. Especially, if we take the proportional cost parameter of upward (downward) adjustment $\alpha=0$ ($\beta=0$), then it captures the optimal consumption problem under only the reference level of the historical consumption valley peak (historical consumption peak). Furthermore, when the upward  adjustment cost parameter $\alpha$ tends to $1/\delta$, the agent will suffer the enormous cost of increasing past spending maximum consumption level comparing to the acquired consumption utility, thus he will never adjust the consumption valley. The case with the downward adjustment cost parameter $\beta\to \infty$ is also similar. Theorem\ref{thm: sensitivity}-(ii) says that when the risk aversion level of the agent is high enough,  the portfolio is almost entirely concentrated in the safe asset and the optimal consumption-wealth ratio should be close to a fraction $r$ of savings, which is consistent with the observations in \cite{GuasoniHubermanR2020MFF}.
%our general formulation \eqref{eq: preference} includes some  

%\begin{theorem}\label{thm:portfolio-wealth-ratio}
%Consider the  optimal feedback control functions $\pi^*(x,h_1,h_2)$ provided in  Corollary \ref{cor: main_res}, then the portfolio-wealth ratio $\pi^*_{\mathrm{ratio}}(x,h_1,h_2):=\frac{\pi^*(x,h_1,h_2)}{x}$ is not large than $\frac{\mu-r}{\sigma^2 \gamma}$ for all $x>0,h_1\geq h_2\geq 0$. Furthermore, there exists  $\hat{x}(h_1,h_2)\in(x_\mathrm{valy}(h_1,h_2),x_\mathrm{peak}(h_1,h_2))$ such that $\pi^*_{\mathrm{ratio}}(x,h_1,h_2)$ is decreasing in $(0,\hat{x}(h_1,h_2))$ and increasing in $(\hat{x}(h_1,h_2),+\infty)$.
%\end{theorem}

The following proposition characterizes the asymptotic properties of optimal consumption and portfolio feedback functions.
\begin{proposition}\label{prop:asymptotic}
Consider the  optimal feedback control functions $\pi^*(x,h_1,h_2)$ and $c^*(x,h_1,h_2)$ provided in  Corollary \ref{cor: main_res}, the following properties hold:
\begin{enumerate}
\item For $h_1\geq h_2\geq 0$, as the wealth level $x\to \infty$, we have that
\begin{align}
\begin{cases}
\displaystyle \bar{\pi}^*_{\mathrm{ratio}}:=\lim_{x\to\infty}\frac{\pi^*(x,h_1,h_2)}{x}=\frac{\mu-r}{\sigma^2 \gamma},\\[1em]
\displaystyle \bar{c}^*_{\mathrm{ratio}}:=\lim_{x\to\infty}\frac{c^*(x,h_1,h_2)}{x}=\frac{\kappa^2(m_1-1)(m_1-\gamma^*)(m_2-\gamma^*)}{2((1-\gamma^*)z_{\alpha}^{m_1-1}-m_1+\gamma^*))}.
\end{cases}
\end{align}
In addition, the function $\alpha \to \bar{c}^*_{\mathrm{ratio}}$ is decreasing and when $\alpha=0$, the limit $\bar{c}^*_{\mathrm{ratio}}$ coincides with the optimal consumption strategy of the classical Merton’s problem in \cite{Mert1971JET}.
\item For $h_1\geq h_2\geq 0$, as the wealth level $x\to 0$, we have that
\begin{align}
\begin{cases}
\displaystyle \underline{\pi}^*_{\mathrm{ratio}}:=\lim_{x\to 0}\frac{\pi^*(x,h_1,h_2)}{x}=\frac{\mu-r}{\sigma^2 \gamma},\\[1em]
\displaystyle \underline{c}^*_{\mathrm{ratio}}:=\lim_{x\to 0}\frac{c^*(x,h_1,h_2)}{x}=\frac{\kappa^2(m_2-1)(m_1-\gamma^*)(m_2-\gamma^*)}{2((1-\gamma^*)z_{\beta}^{m_2-1}-m_2+\gamma^*))}.
\end{cases}
\end{align}
In addition, the function $\beta \to \underline{c}^*_{\mathrm{ratio}}$ is increasing and when $\beta = 0$, the limit $\underline{c}^*_{\mathrm{ratio}}$ coincides with the optimal consumption strategy of the classical Merton’s problem in \cite{Mert1971JET}.
\end{enumerate}
\end{proposition}

Proposition \ref{prop:asymptotic} indicates that when the wealth level is sufficiently high, the agent is more concerned about the cost of updating the past spending maximum of consumption, thus the asymptotic optimal portfolio-wealth ratio depends on the upward adjustment cost parameter $\alpha$; while the case when the wealth level is sufficiently low is in contrast. Furthermore, to avoid high adjustment cost, when the wealth level is sufficiently high (resp. low) , the agent behaves less (resp. more) aggressively in consumption plan compared with the Merton solution. 

On the other hand, Proposition \ref{prop:asymptotic} shows that when the wealth level is sufficiently or low, the impact of adjustment costs with respect to the reference levels of the past spending maximum and minimum on the portfolio decision of the agent becomes negligible. The asymptotic behavior of the optimal portfolio-wealth ratio coincides with the solution of the classical Merton problem. However, the following proposition implies that the agent will become more risk-averse and put less wealth into the risky asset to maintain the current reference level of consumption and avoid the adjustment costs.

\begin{proposition}\label{prop:portfolio}
  The portfolio-wealth ratio $\pi^*_{\mathrm{ratio}}(x,h_1,h_2):=\frac{\pi^*(x,h_1,h_2)}{x}$ is not larger than $\frac{\mu-r}{\sigma^2 \gamma}$ for all $x>0,h_1\geq h_2\geq 0$.
\end{proposition}

%\subsection{Long-run properties}
Next, we also present some long-run properties of the optimal wealth process.
\begin{corollary}\label{thm: long_run}
\begin{enumerate}
%\item If $\delta-r-\frac{\kappa^2}{2} < 0$, the long-run average time spent in the peak region is $1 - z_\alpha^{1+\frac{2(r-\delta)}{\kappa^2}}$, and the long-run average time spent in the valley region is 0.
%    If $\delta-r-\frac{\kappa^2}{2} > 0$, the long-run average time spent in the peak region is $1 - z_\beta^{1+\frac{2(r-\delta)}{\kappa^2}}$, and the long-run average time spent in the peak region is 0.
%    If $\delta-r-\frac{\kappa^2}{2} = 0$, the long-run average times spent in the peak/valley regions are both 0.
\item The long-run average time spent in the peak region is $\big(1 - z_\alpha^{1+\frac{2(r-\delta)}{\kappa^2}}\big)_+$, and the long-run average time spent in the valley region is $\big(1 - z_\beta^{1+\frac{2(r-\delta)}{\kappa^2}}\big)_+$.
    Here $(\cdot)_+$ is defined by $\max(\cdot, 0)$.
\item Starting from initial wealth $x$ and references $h_1$, $h_2$ with $f(x, h_1,h_2)\in\big[z_\alpha h_1^{-\gamma}, z_\beta h_2^{-\gamma}\big]$,
    let us define the first time to adjust the consumption valley $\tau_\mathrm{gloom} = \inf\{t\geq0: X_t = \underline{x}(H_1^*, H_2^*)\}$.
    Then
    $$
    P(\tau_\mathrm{gloom} < \infty|x, h_1, h_2) =
    \begin{cases}
    1, &\mbox{if } \delta-r-\frac{\kappa^2}{2} \geq  0,\\
    \bigg(\frac{z_\beta h_2^{-\gamma}}{f(x, h_1, h_2)}\bigg)^{1 + \frac{2(r-\delta)}{\kappa^2}}, &\mbox{if } \delta-r-\frac{\kappa^2}{2} < 0.
    \end{cases}
    $$
    That is, if $\delta-r-\frac{\kappa^2}{2} < 0$, the spending valley will with positive probability fail to be adjusted.
    In addition, if $\delta-r-\frac{\kappa^2}{2} \geq 0$,
    $$
    \E_{x, h_1, h_2}[\tau_\mathrm{gloom}]
    = \frac{1}{\delta-r-\frac{\kappa^2}{2}}\log\bigg(\frac{z_\beta h_2^{-\gamma}}{f(x, h_1, h_2)}\bigg).
    $$
    In particular, starting from lavish $\big(f(x, h_1, h_2) = z_\alpha h_1^{-\gamma}\big)$,
    $$
    \E_{x, h_1, h_2}[\tau_\mathrm{gloom}] = \frac{1}{\delta-r-\frac{\kappa^2}{2}}\log\bigg(\frac{z_\beta h_1^{\gamma}}{z_\alpha h_2^{\gamma}}\bigg).
    $$

\item Starting from initial wealth $x$ and references $h_1$, $h_2$ with $f(x, h_1,h_2)\in\big[z_\alpha h_1^{-\gamma}, z_\beta h_2^{-\gamma}\big]$,
    let us define the first time to adjust the consumption peak $\tau_\mathrm{lavs} = \inf\{t\geq0: X_t = \bar{x}(H_1^*, H_2^*)\}$.
    Then
    $$
    P(\tau_\mathrm{lavs} < \infty|x, h_1, h_2) =
    \begin{cases}
    1, &\mbox{if } \delta-r-\frac{\kappa^2}{2} \leq  0,\\
    \bigg(\frac{f(x, h_1, h_2)h_1^{\gamma}}{z_\alpha}\bigg)^{-\big(1 + \frac{2(r-\delta)}{\kappa^2}\big)}, &\mbox{if } \delta-r-\frac{\kappa^2}{2} > 0.
    \end{cases}
    $$
    That is, if $\delta-r-\frac{\kappa^2}{2} > 0$, the spending valley will with positive probability fail to be adjusted.
    In addition, if $\delta-r-\frac{\kappa^2}{2} \leq 0$,
    $$
    \E_{x, h_1, h_2}[\tau_\mathrm{lavs}]
    = -\frac{1}{\delta-r-\frac{\kappa^2}{2}}\log\bigg(\frac{f(x, h_1, h_2)h_1^{\gamma}}{z_\alpha}\bigg).
    $$
    In particular, starting from gloom $\big(f(x, h_1, h_2) = z_\beta h_2^{-\gamma}\big)$,
    $$
    \E_{x, h_1, h_2}[\tau_\mathrm{lavs}] = -\frac{1}{\delta-r-\frac{\kappa^2}{2}}\log\bigg(\frac{z_\beta h_1^{\gamma}}{z_\alpha h_2^{\gamma}}\bigg).
    $$
\end{enumerate}
\end{corollary}
%\begin{theorem}
%The long-run return on the optimal portfolio, $\tilde{r} = \lim\limits_{T\rightarrow+\infty} \frac{1}{T}\int_0^T \big(r + (\mu-r) \frac{\pi_t^*}{X_t}\big)dt$, is
%$$
%\E_{x, h_1, h_2}[\tilde{r}] = ...
%$$
%\end{theorem}
%
%\begin{theorem}
%The cost of ignoring loss aversion in the optimal policy is:
%\end{theorem}

\section{Numerical Examples}\label{sec: numerical}
In this section, we present some examples of numerical analysis using the closed-form value function and feedback optimal controls in Corollary \ref{cor: main_res} and discuss some interesting financial implications.

\begin{figure}[h]
\centering
\includegraphics[width=5in]{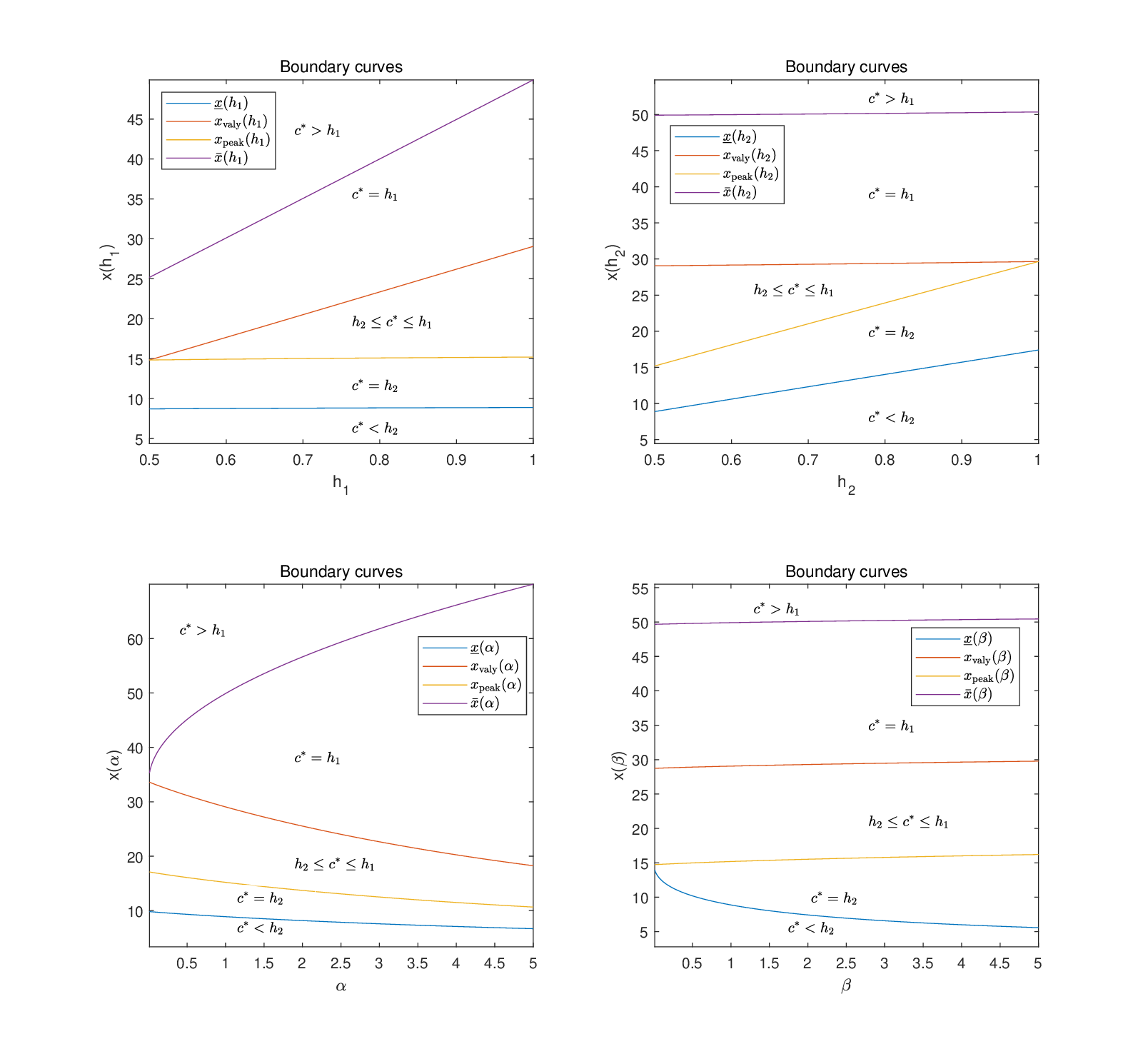}
\vspace{-0.3in}
\caption{{\small Boundary curves}}
\label{fig: boundary}
\end{figure}

Let us choose the following values of the model parameters: $r = 0.0063$, 
$\delta = 0.2$,  
$\mu = 0.15$,  
$\sigma = 0.2$,  
$\gamma = 0.6$,
$\alpha = 1$,
$\beta = 1$,
$h_1 = 1$ and
$h_2 = 0.2$. 
In the following sensitivity analysis, we only change the value of one parameter (while keeping other parameters fixed) to show some sensitivity results with respect to that parameter.

Figure \ref{fig: boundary} graphs all boundary curves $x_\mathrm{gloom}(h_1,h_2;\alpha,\beta)$, $x_\mathrm{valy}(h_1,h_2;\alpha,\beta)$, $x_\mathrm{peak}(h_1,h_2;\alpha,\beta)$, and $x_\mathrm{lavs}(h_1,h_2;\alpha,\beta)$ as functions of $h_1, h_2, \alpha$, and $\beta$.
Although the four curves represent complex nonlinear equations for $h_1$ and $h_2$, the upper-left and upper-right graphs indicate that they are both monotonically increasing with respect to $h_1$ and $h_2$.
On one hand, this is consistent with the results in Proposition \ref{prop: boundary_deri}, which also helps us visualize the curvature of the curves, albeit barely visible in the image.
On the other hand, this result aligns with intuition, indicating that a higher reference level (regardless of which one) would make the agent consider larger wealth thresholds to trigger a change in consumption patterns.
In terms of strategy, we essentially only need to consider the case where $x_\mathrm{gloom}(h_1,h_2) \leq x \leq x_\mathrm{lavs}(h_1,h_2)$ because once $x$ deviates from the range between the two curves, the agent will immediately update $h_1$ or $h_2$, generating new boundaries $x_\mathrm{gloom}(h_1,h_2)$ and $x_\mathrm{lavs}(h_1,h_2)$ that still encompass the wealth $x$.

Upon observing the two bottom-left and bottom-right graphs, we notice that $x_\mathrm{lavs}(h_1,h_2;\alpha,\beta)$ is monotonically increasing with respect to $\alpha$, while the remaining three curves are monotonically decreasing with respect to $\alpha$.
This means that if the agent has a more conservative attitude towards adjusting the spending maximum (with larger $\alpha$), she is less likely to adjust towards the lower consumption valley and more likely to switch from a gloomy consumption state to normal consumption ($h_2 < c < h_1$).
Furthermore, it is also easier for her to transition from normal consumption to a higher consumption mode.
In other words, except for updating lavish consumption, the agent's consumption at other times becomes more aggressive.
This is reasonable because once the agent becomes more cautious in updating the spending maximum, it can create a certain amount of ``wealth redundancy".
This ``wealth redundancy" can then support more aggressive consumption at other times and yield corresponding utility gains.
Similarly, $x_\mathrm{gloom}(h_1,h_2;\alpha,\beta)$ is monotonically decreasing with respect to $\beta$, while the remaining three curves are monotonically increasing with respect to $\beta$.
This means that if the agent has a more conservative attitude towards adjusting the spending minimum (with larger $\beta$), she is less likely to adjust towards the higher consumption peak.
However, she is also less likely to switch from a gloomy consumption state to normal consumption ($h_2 < c < h_1$) and less likely to transition from normal consumption to a lavish consumption mode.
In other words, except for updating gloomy consumption, the agent's consumption at other times becomes more conservative.
This is also reasonable because if the agent is more cautious in adjusting the spending minimum, it makes it easier for wealth to experience so-called ``shortfalls".
These ``shortfalls" need to be compensated for through a more conservative consumption strategy in normal times.

In particular, the boundary curves discussed above can partially explain a reality: pursuing an aggressive or overly conservative consumption strategy does not lead to long-term happiness.
Engaging in high levels of consumption can bring about a significant psychological burden, and such behavior is often unsustainable, resulting in diminishing feelings of happiness over time.
On the other hand, adopting a low-consumption behavior can foster frugal habits but may also lead to missed opportunities for experiencing the beauty of the world.
Therefore, wise individuals consider past references, whether they are peaks or valleys, to help them adjust their strategies in response to the unpredictability of life.
If the agent has experienced ups and downs and has already reached a high level of $h_1$ and a low level of $h_2$, her behavior is less influenced by psychological factors, enabling her to make decisions that truly benefit themselves.
This aligns with the saying, ``I've seen the highs and lows, and both have enriched me".

%\subsection{Sensitivity analysis}
\begin{figure}[htbp]
\centering
\includegraphics[width=5in]{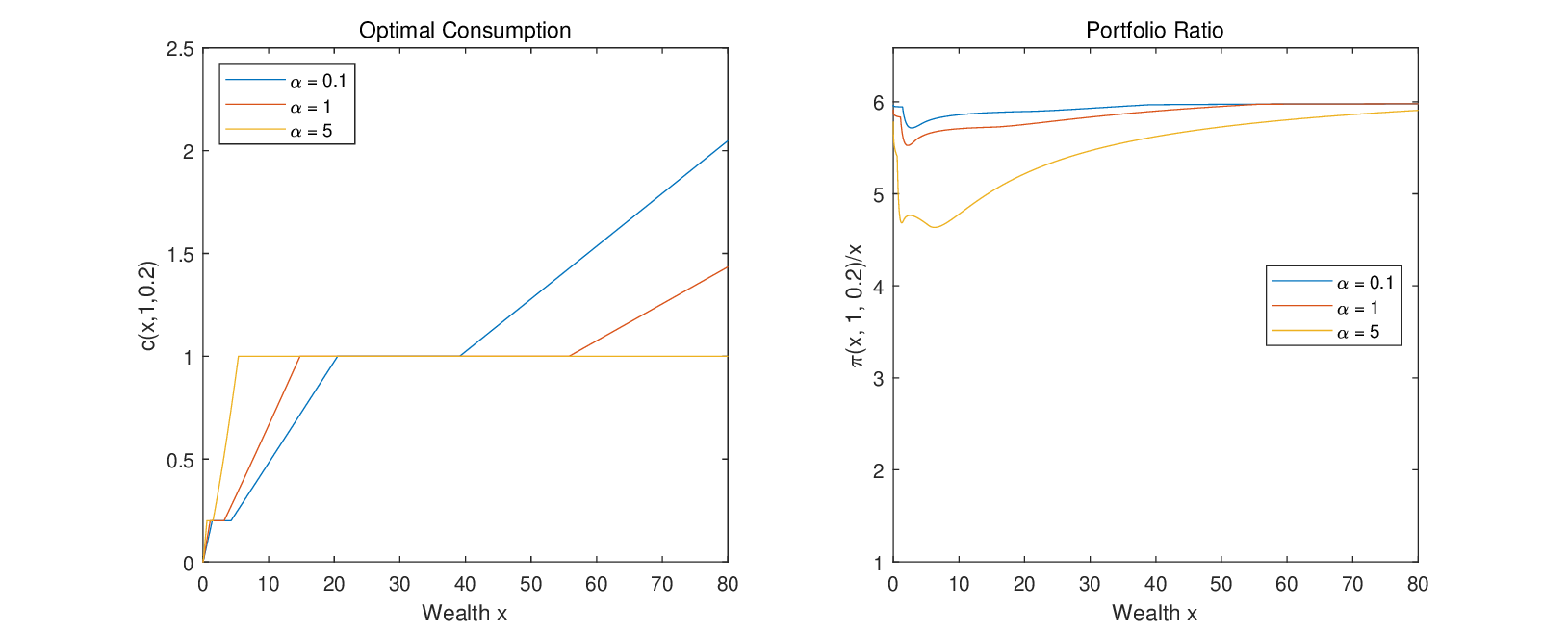}
\vspace{-0.2in}
\caption{{\footnotesize Sensitivity of $\alpha$}}
\label{fig: sen_alpha}
\end{figure}
Let us consider the sensitivity with respect to the adjustment parameter $\alpha\in(0, 1/\delta)$ by plotting in Figure \ref{fig: sen_alpha} some comparison graphs of the value function and the optimal feedback controls (consumption and portfolio).
The left panel once again confirms that the value function is decreasing with respect to $\alpha$.
This can also be clearly seen in the problem formulation \eqref{eq: preference}, where for any fixed admissible strategy, a larger $\alpha$ leads to a greater loss of utility when adjusting the past spending maximum.
From the middle panel, we observe that $x_\mathrm{lavs}(h_1,h_2;\alpha,\beta)$ increases with respect to $\alpha$, while $x_\mathrm{gloom}(h_1,h_2;\alpha,\beta)$, $x_\mathrm{valy}(h_1,h_2;\alpha,\beta)$, and $x_\mathrm{peak}(h_1,h_2;\alpha,\beta)$ decrease with respect to $\alpha$.
Moreover, for a fixed wealth $x$, we find that within the range where $x$ is less than $x_\mathrm{lavs}(h_1,h_2;\alpha,\beta)$, the optimal consumption increases with respect to $\alpha$.
This contradicts our intuition a little, as we tend to assume that conservative individuals will always be conservative.
However, there is a rationale for this, as conservative individuals have the capacity to be more aggressive in certain situations.
When the wealth is sufficiently small, the agent can only choose a gloom consumption level.
In the range $x < x_\mathrm{gloom}(h_1,h_2;\alpha,\beta)$, higher $\alpha$ values ultimately dominate lower $\alpha$ values in terms of consumption.
When the wealth is sufficiently large, the agent can freely choose a lavish consumption level without constraints.
In the range $x > x_\mathrm{lavs}(h_1,h_2;\alpha,\beta)$, lower $\alpha$ values ultimately dominate higher $\alpha$ values in terms of consumption.
We can also observe from the right panel that for fixed $x > 0$, $\pi^*(x, h_1, h_2;\alpha,\beta)$ seems decreasing in $\alpha$, especially in region $x_\mathrm{valy}(h_1,h_2;\alpha,\beta) < x < x_\mathrm{peak}(h_1,h_2;\alpha,\beta)$.
This is because as $\alpha$ increases, consumption in this region becomes more aggressive, prompting the agent to strategically withdraw some investment portfolios to support the consumption plan.
Of particular interest is a fascinating phenomenon where, when wealth $x$ is sufficiently large or sufficiently small, the optimal portfolios corresponding to different $\alpha$'s exhibit almost linear variations and overlap with each other.
Proposition \ref{prop:asymptotic} confirms this result and further emphasizes that these overlapping optimal portfolios degenerate into solutions of the classical Merton's problem, independent of $\alpha$ and $\beta$ values.
This implies that when wealth $x$ is sufficiently large or sufficiently small, regardless of the agent's attitude towards spending references, her portfolio strategy will be hardly affected.
This phenomenon partially aligns with our intuition because when $x$ is sufficiently small, $h_2$ decreases correspondingly, making the influences of $h_1$ and $\alpha$ on the model gradually negligible.
Similarly, when $x$ is sufficiently large, $h_1$ increases correspondingly, while the influences of $h_2$ and $\beta$ on the model can be gradually ignored.
In these cases, the homogeneous property leads to optimal portfolios and consumption exhibiting a tendency towards linear variations with respect to wealth.
Surprisingly, the asymptotic ratio of the optimal portfolio demonstrates unusual robustness, being independent of $\alpha$ and $\beta$.
This may be due to the fact that if one solely relies on reducing funds from the portfolio to increase consumption, it will significantly impact wealth growth, thereby affecting utility.

\begin{figure}[htbp]
\centering
\includegraphics[width=5in]{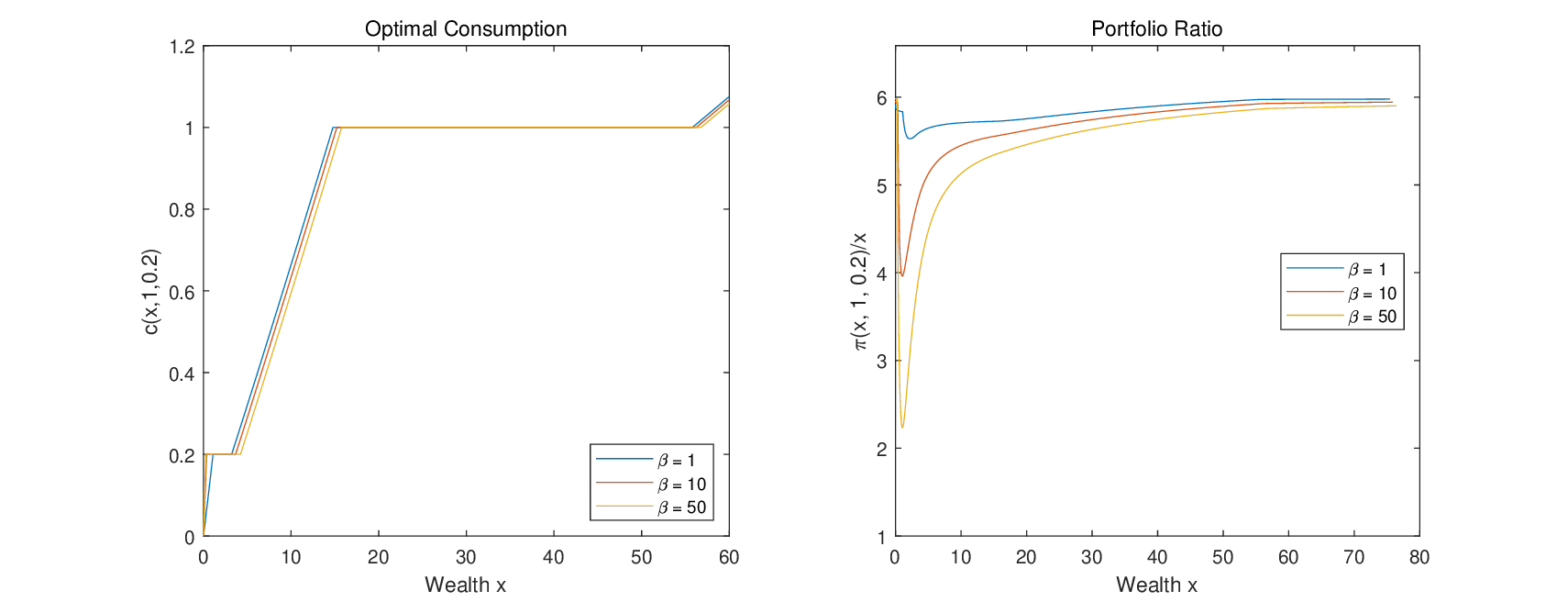}
\vspace{-0.2in}
\caption{{\footnotesize Sensitivity of $\beta$}}
\label{fig: sen_beta}
\end{figure}

Let us next consider the sensitivity with respect to the adjustment parameter $\beta>0$ by plotting in Figure \ref{fig: sen_beta} some comparison graphs of the value function and the optimal feedback controls.
The left panel once again confirms that the value function is decreasing with respect to $\beta$.
This can also be clearly seen in the problem formulation \eqref{eq: preference}, where for any fixed admissible strategy, a larger $\beta$ leads to a greater loss of utility when adjusting the past spending minimum.
From the middle panel, we observe that $x_\mathrm{gloom}(h_1,h_2;\alpha,\beta)$ decreases with respect to $\beta$, while $x_\mathrm{valy}(h_1,h_2;\alpha,\beta)$, $x_\mathrm{peak}(h_1,h_2;\alpha,\beta)$, and $x_\mathrm{lavs}(h_1,h_2;\alpha,\beta)$ increase with respect to $\beta$.
Moreover, for a fixed wealth $x$, we find that within the range where $x$ is greater than $x_\mathrm{gloom}(h_1,h_2;\alpha,\beta)$, the optimal consumption increases with respect to $\beta$.
This coincides with the results for sensitivity analysis of $\alpha$, that conservative individuals have the capacity to be more aggressive in certain situations.
In the range $x < x_\mathrm{gloom}(h_1,h_2;\alpha,\beta)$, larger $\beta$ values ultimately dominate lower $\beta$ values in terms of consumption.
We can also observe that from the right panel that for fixed $x > 0$, $\pi^*(x, h_1, h_2;\alpha,\beta)$ seems decreasing in $\beta$, especially in region $x_\mathrm{valy}(h_1,h_2;\alpha,\beta) < x < x_\mathrm{peak}(h_1,h_2;\alpha,\beta)$.
This is because as $\beta$ increases, the utility loss caused by adjusting the consumption valley is more unacceptable, leading to a preference for a more conservative consumption strategy.
This involves allocating more funds to the risk-free asset, thereby reducing the investment in the portfolio.

Finally, we study the sensitivity with respect to the drift parameter $\mu$ in Figures \ref{fig: sen_mu}. From the left panel, it can be observed that the value function is increasing with respect to $\mu$, which is entirely reasonable. Since the optimal portfolio is always greater than 0, a higher drift leads to larger potential wealth and consequently higher utility.
Similarly, we find that each boundary curve is increasing with respect to $\mu$.
Consequently, a higher $\mu$ results in a smaller optimal consumption.
This aligns with the results of the classical Merton's problem and provides strong evidence for the comparison between quantitative models and intuition.
On one hand, during a bull market, people tend to be more optimistic about the future, which may make them more willing to engage in higher consumption.
On the other hand, due to favorable market development, people may be more inclined to invest more of their wealth in the market, even by reducing expenses to increase the proportion of risky investments.
The numerical results indicate that, in our model, agents' expectations of investment returns outweigh the present happiness derived from current consumption.
Furthermore, the right panel also indicates that as the drift increases, people do indeed increase their level of investment in risky assets.
From Figure \ref{fig: sen_alpha}, for the same level of wealth $x$, we observe that the optimal investment portfolio tends to decrease with respect to alpha and beta in certain cases, but it consistently exhibits a monotonic increase with respect to $\mu$.
This may partially explain the puzzle of the observed equity premium.

%\begin{figure}[htbp]
%\centering
%\includegraphics[width=\textwidth]{./Figures/sen_gamma.eps}
%\vspace{-0.4in}
%\caption{{\footnotesize Sensitivity of $\gamma$}}
%\label{fig: sen_gamma}
%\end{figure}

%\begin{figure}[htbp]
%\centering
%\includegraphics[width=\textwidth]{./Figures/sen_h1.eps}
%\vspace{-0.4in}
%\caption{{\footnotesize Sensitivity of $h_1$}}
%\label{fig: sen_h1}
%\end{figure}

%\begin{figure}[htbp]
%\centering
%\includegraphics[width=\textwidth]{./Figures/sen_h2.eps}
%\vspace{-0.4in}
%\caption{{\footnotesize Sensitivity of $h_2$}}
%\label{fig: sen_h2}
%\end{figure}

\begin{figure}[htbp]
\centering
\includegraphics[width=5in]{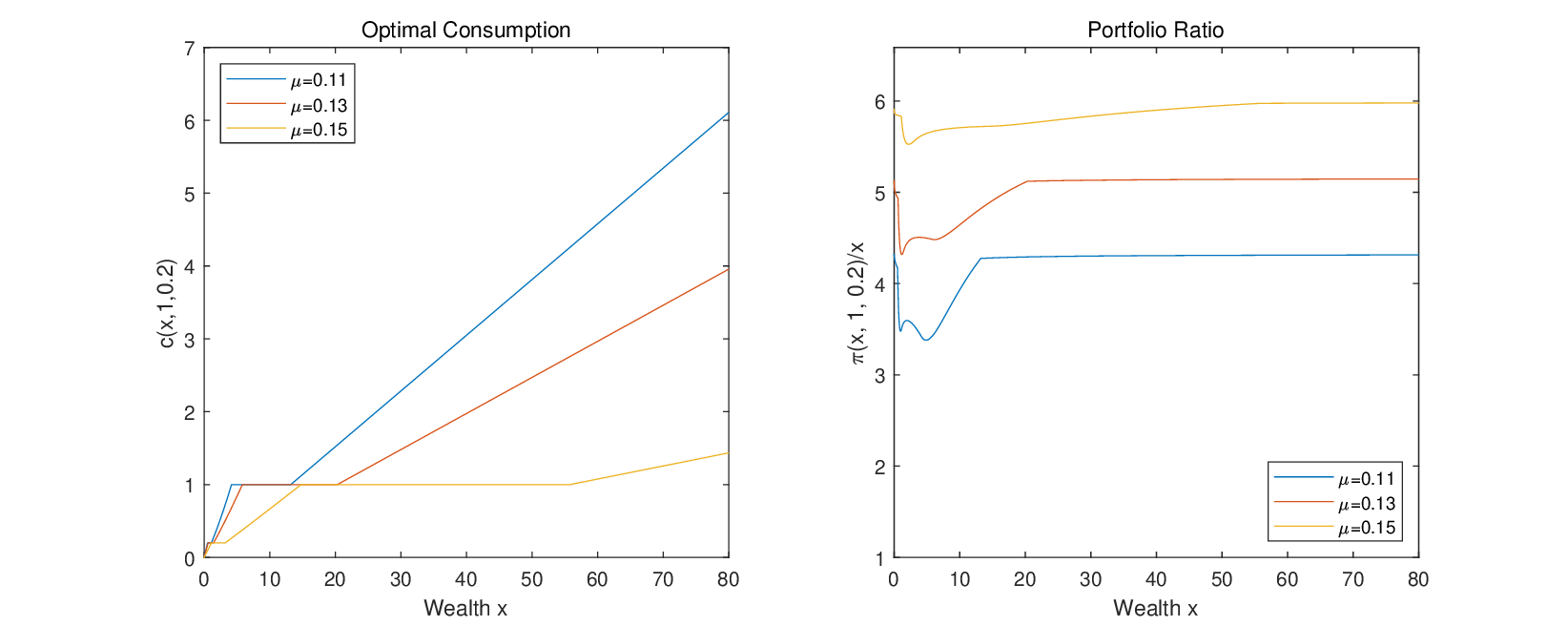}
\vspace{-0.2in}
\caption{{\footnotesize Sensitivity of $\mu$}}
\label{fig: sen_mu}
\end{figure}

%\begin{figure}[htbp]
%\centering
%\includegraphics[width=\textwidth]{./Figures/sen_sigma.eps}
%\vspace{-0.4in}
%\caption{{\footnotesize Sensitivity of $\sigma$}}
%\label{fig: sen_sigma}
%\end{figure}

%\subsection{A simulated path}
%
%\begin{figure}[htbp]
%\centering
%\includegraphics[width=\textwidth]{./Figures/Path_eg.eps}
%\vspace{-0.4in}
%\caption{{\footnotesize Simulated path}}
%\label{fig: path_eg}
%\end{figure}

\begin{figure}[htbp]
\centering
\includegraphics[width=7in]{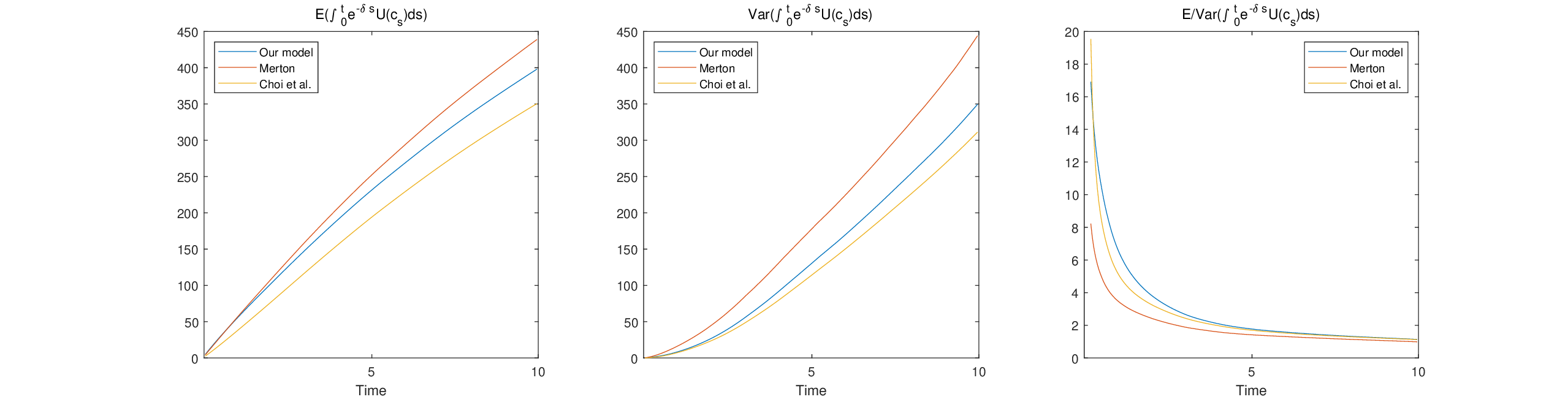}
\vspace{-0.2in}
\caption{{\footnotesize Expectation and Variance of Consumption Process}}
\label{fig: Exp/var}
\end{figure}

Figure \ref{fig: Exp/var} compares expected cumulative consumption utility $\E[\int_0^t e^{-\delta s}U(c_s)ds]$, variance of the cumulative utility $\mbox{Var}(\int_0^t e^{-\delta s}U(c_s)ds)$, and the ratio of expected cumulative utility and variance $\E[\int_0^t e^{-\delta s}U(c_s)ds]/\mbox{Var}(\int_0^t e^{-\delta s}U(c_s)ds)$ of our model, classical Merton's model, and intertemporal loss aversion model in \cite{ChoiJeonKoo2022JET}, respectively.
We select $c_0 = 0.5$ for the model in \cite{ChoiJeonKoo2022JET}.
%Due to our model parameters ($\delta-r-\frac{\kappa^2}{2} = -0.01<0$), the drift of $Y_t(y^*)$ is less than 0, resulting in an increasing trend in the mean and variance of the wealth process and consumption, as depicted in the left and middle panels.
From the left panel, we observe that our model's consumption utility on consumption dominates that of \cite{ChoiJeonKoo2022JET}, indicating that our model is more aggressive in consumption behavior.
%Moreover, as time goes by, our utility growth curve is concave-convex, suggesting a slower initial growth in consumption followed by accelerated growth.
In addition, our consumption utility outperforms that of Merton's problem in the beginning.
As time goes by, Merton's consumption gradually surpasses ours over time.
This is because initially, with relatively small wealth, we tend towards maintaining higher consumption levels.
However, as time progresses and wealth grows, we lean towards more conservative consumption levels, gradually being overtaken by Merton's approach.
In the middle panel, the results align perfectly with our intuition.
Merton's consumption exhibits the highest volatility, \cite{ChoiJeonKoo2022JET}'s consumption exhibits the lowest volatility, and our consumption lies in between.
This is because Merton's problem focuses all its efforts on maximizing the utility benefits from consumption, while Choi's problem primarily concentrates on regulating consumption with minimal changes.
In contrast, our model represents a compromise between the two approaches.
We posit that changes in consumption within a certain reference condition do not elicit psychological feelings of loss, enabling us to achieve greater utility.
However, upon crossing this reference condition, we reasonably acknowledge that agents will experience a sense of loss.
This balance ensures that while maintaining a certain expected utility, we also effectively control the fluctuations in utility.
From the right panel, it is evident that Merton's problem has the lowest utility-variance ratio $\E[\int_0^t e^{-\delta s}U(c_s)ds]/\mbox{Var}(\int_0^t e^{-\delta s}U(c_s)ds)$.
Here, we have not reported the case where $t$ is very small.
This is because our model and Choi's model \cite{ChoiJeonKoo2022JET} maintain flat levels of consumption without making significant changes, leading to a variance approaching 0 at the beginning.
As a result, reporting the ratio becomes meaningless and thus we exclude this portion.
As consumption begins to change, our ratio consistently dominates that of Merton's model, with Choi's model being the highest initially and then quickly overtaken by our model.
From this observation, we can see that our model indeed combines the strengths of both models and this is reflected in the metrics on the right panel.
We have retained a relatively high consumption utility while also ensuring a more robust consumption behavior with lower variance.

\appendix
\section{Proofs}\label{sec:Appd}
\subsection{Proof of Proposition \ref{prop: solu_dual_value}}
\begin{proof}
It is straightforward to see that the linear ODE \eqref{eq: HJB_dual} admits the general solution
$$
v(y,h_1,h_2) =
\left\{\begin{aligned}
&C_1(h_1,h_2)y^{m_1} + C_2(h_1,h_2)y^{m_2} + \frac{h_1^{1-\gamma}}{\delta(1-\gamma)} - \frac{yh_1}{r},  &\mbox{ if }& z_\alpha h_1^{-\gamma} \leq y < h_1^{-\gamma}, \\
&C_3(h_1,h_2)y^{m_1} + C_4(h_1,h_2)y^{m_2} + \frac{2}{\kappa^2 \gamma^*(\gamma^*-m_1)(\gamma^*-m_2)}y^{\gamma^*},  &\mbox{ if }& h_1^{-\gamma} \leq y \leq h_2^{-\gamma}, \\
&C_5(h_1,h_2)y^{m_1} + C_6(h_1,h_2)y^{m_2} + \frac{h_2^{1-\gamma}}{\delta(1-\gamma)} - \frac{yh_2}{r},  &\mbox{ if }& h_2^{-\gamma} < y \leq z_\beta h_2^{-\gamma}.
\end{aligned}\right.
$$
The smooth conditions at $y = h_1^{-\gamma}$ and $y = h_2^{-\gamma}$ give us the equations of $C_1(c,h), \cdots, C_6(c,h)$ that
\begin{align*}
&\big(C_{1}(h_1, h_2) - C_3(h_1,h_2)\big)\big(h_1^{-\gamma}\big)^{m_1} + \big(C_{2}(h_1, h_2) - C_4(h_1,h_2)\big)\big(h_1^{-\gamma}\big)^{m_2} \\
= &- \frac{h_1^{1-\gamma}}{\delta(1-\gamma)} + \frac{h_1^{1-\gamma}}{r}+ \frac{2}{\kappa^2\gamma^*(\gamma^*-m_1)(\gamma^*-m_2)}\big(h_1^{-\gamma}\big)^{\gamma^*}, \\
&m_1\big(C_{1}(h_1, h_2) - C_3(h_1,h_2)\big)\big(h_1^{-\gamma}\big)^{m_1-1} + m_2\big(C_{2}(h_1, h_2) - C_4(h_1,h_2)\big)\big(h_1^{-\gamma}\big)^{m_2-1} \\
= & \frac{h_1}{r}+ \frac{2}{\kappa^2(\gamma^*-m_1)(\gamma^*-m_2)}\big(h_1^{-\gamma}\big)^{\gamma^*-1}, \\
&\big(C_{3}(h_1, h_2) - C_5(h_1,h_2)\big)\big(h_2^{-\gamma}\big)^{m_1} + \big(C_{4}(h_1, h_2) - C_6(h_1,h_2)\big)\big(h_2^{-\gamma}\big)^{m_2} \\
= &\frac{h_2^{1-\gamma}}{\delta(1-\gamma)} - \frac{h_2^{1-\gamma}}{r} - \frac{2}{\kappa^2\gamma^*(\gamma^*-m_1)(\gamma^*-m_2)}\big(h_2^{-\gamma}\big)^{\gamma^*}, \\
&m_1\big(C_{3}(h_1, h_2) - C_5(h_1,h_2)\big)\big(h_2^{-\gamma}\big)^{m_1-1} + m_2\big(C_{4}(h_1, h_2) - C_6(h_1,h_2)\big)\big(h_2^{-\gamma}\big)^{m_2-1} \\
= & -\frac{h_2}{r} - \frac{2}{\kappa^2(\gamma^*-m_1)(\gamma^*-m_2)}\big(h_2^{-\gamma}\big)^{\gamma^*-1}. 
\end{align*}

We can consider them as a linear system of $C_1(h_1,h_2)-C_3(h_1,h_2)$, $C_3(h_1,h_2)-C_5(h_1,h_2)$, $C_2(h_1,h_2)-C_4(h_1,h_2)$, and $C_4(h_1,h_2)-C_6(h_1,h_2)$.
By solving the linear system using the fact that $m_1+m_2 = 1 + \frac{2(r-\delta)}{\kappa^2}$, $m_1m_2 = -\frac{2\delta}{\kappa^2}$, and $\gamma^* = -\frac{1-\gamma}{\gamma}$, we obtain
\begin{equation}\label{eq: C_diff}
\begin{aligned}
C_1(h_1,h_2)-C_3(h_1,h_2) &= \frac{2(1-\gamma^*)}{\kappa^2(m_1-m_2)m_1(m_1-1)(m_1-\gamma^*)}h_1^{1 + \gamma(m_1-1)}, \\
C_2(h_1,h_2)-C_4(h_1,h_2) &= \frac{2(\gamma^*-1)}{\kappa^2(m_1-m_2)m_2(m_2-1)(m_2-\gamma^*)}h_1^{1 + \gamma(m_2-1)}, \\
C_3(h_1,h_2)-C_5(h_1,h_2) &= \frac{2(\gamma^*-1)}{\kappa^2(m_1-m_2)m_1(m_1-1)(m_1-\gamma^*)}h_2^{1 + \gamma(m_1-1)}, \\
C_4(h_1,h_2)-C_6(h_1,h_2) &= \frac{2(1-\gamma^*)}{\kappa^2(m_1-m_2)m_2(m_2-1)(m_2-\gamma^*)}h_2^{1 + \gamma(m_2-1)}. \\
\end{aligned}
\end{equation}
%
%If we consider $yh_2^{\gamma} \in (1, \by(h_1,h_2)h_2^{\gamma})$ and fix $h_2/h_1$, then as $h_2\rightarrow0$ and thus $y\rightarrow\infty$, $h_1\rightarrow0$, we have $v(y, h_1, h_2) = h_2^{1-\gamma}v(yh_2^{\gamma}, h_1/h_2, 1)\rightarrow0$, indicating that $C_5(h_1, h_2) \rightarrow 0$, that is, $C_5(0,0) = 0$.
%If we consider $yh_1^{\gamma}\in (\uy(h_1,h_2)h_1^{\gamma}, 1)$ and fix $h_1/h_2$, then as $h_1\rightarrow\infty$ and thus $y\rightarrow0$, $h_2\rightarrow\infty,$ we have $v(y, h_1, h_2) = h_1^{1-\gamma}v(yh_1^{\gamma}, 1, h_2/h_1)$, indicating that $C_2(h_1, h_2) = O\big(h_1^{1 + \gamma(m_2-1)}\big)\rightarrow 0$, that is, $C_2(+\infty, +\infty) = 0$.

As $h_1\rightarrow +\infty$ and $h_2\rightarrow0$, according to condition \eqref{eq: dual_merton}, we obtain $C_i(+\infty, 0) = 0$ for $i=1,\cdots,6$.
Together with the homogenous property and \eqref{eq: C_diff}, we guess that $C_3(h_1,h_2)$ and $C_5(h_1,h_2)$ do not depend on $h_1$, and $C_2(h_1, h_2)$ and $C_4(h_1, h_2)$ do not depend on $h_2$.
Therefore, $C_{3,h_1}(h_1, h_2) = C_{5, h_1}(h_1, h_2) = 0$, $C_{2, h_2}(h_1, h_2) = C_{4, h_2}(h_1,h_2) = 0$.
In this case, we obtain by \eqref{eq: C_diff} that
$$
\begin{aligned}
C_{1, h_1}(h_1,h_2) &= C_{1, h_1}(h_1,h_2) - C_{3, h_1}(h_1,h_2) = \frac{2}{\kappa^2(m_1-m_2)m_1(m_1-1)}h_1^{\gamma(m_1-1)}, \\
C_{6, h_2}(h_1, h_2) &= -\big(C_{4, h_2}(h_1, h_2) - C_{6, h_2}(h_1, h_2)\big)
= -\frac{2}{\kappa^2(m_1-m_2)m_2(m_2-1)}h_2^{\gamma(m_2-1)}.
\end{aligned}
$$

In addition, according to the boundary conditions that one updates the past spending maximum/minimum \eqref{eq: dual_boundary_adjust1} \eqref{eq: dual_boundary_adjust2}, we derive the following equations:
\begin{equation*}
\begin{aligned}
C_{1,h_1}(h_1, h_2)\uy(h_1,h_2)^{m_1} + C_{2,h_1}(h_1, h_2)\uy(h_1,h_2)^{m_2} + \frac{h_1^{-\gamma}}{\delta} - \frac{\uy(h_1,h_2)}{r}
&= \alpha h_1^{-\gamma}, \\
m_1C_{1,h_1}(h_1, h_2)\uy(h_1,h_2)^{m_1-1} + m_2C_{2,h_1}(h_1, h_2)\uy(h_1,h_2)^{m_2-1} - \frac{1}{r}
&= 0,\\
%m_1(m_1-1)C_{1,h_1}(h_1, h_2)\uy(h_1,h_2)^{m_1-2} + m_2(m_2-1)C_{2,h_1}(h_1, h_2)\uy(h_1,h_2)^{m_2-2}
%&= 0,\\
C_{5,h_2}(h_1, h_2)\by(h_1,h_2)^{m_1} + C_{6,h_2}(h_1, h_2)\by(h_1,h_2)^{m_2} + \frac{h_2^{-\gamma}}{\delta} - \frac{\by(h_1,h_2)}{r}
&= -\beta h_2^{-\gamma}, \\
m_1C_{5,h_2}(h_1, h_2)\by(h_1,h_2)^{m_1-1} + m_2C_{6,h_2}(h_1, h_2)\by(h_1,h_2)^{m_2-1} - \frac{1}{r}
&= 0.
%m_1(m_1-1)C_{5,h_2}(h_1, h_2)\by(h_1,h_2)^{m_1-2} + m_2(m_2-1)C_{6,h_2}(h_1, h_2)\by(h_1,h_2)^{m_2-2}
%&= 0.\\
\end{aligned}
\end{equation*}
The first two equations are a system of two unknown parameters $C_{2,h_1}(h_1,h_2)$ and $\uy(h_1,h_2)$, and the last two equations are a system of two unknown parameters $C_{5, h_2}(h_1, h_2)$ and $\by(h_1,h_2)$.
Solving theses two systems, we obtain
$$
\begin{aligned}
C_{2, h_1}(h_1, h_2) = \frac{1}{m_1-m_2}\bigg(\frac{m_1(\alpha\delta-1)}{\delta}z_\alpha^{-m_2} + \frac{m_1-1}{r}z_\alpha^{1-m_2}\bigg) h_1^{\gamma(m_2-1)},\\
C_{5, h_2}(h_1, h_2) = \frac{1}{m_1-m_2}\bigg(\frac{m_2(\beta\delta+1)}{\delta}z_\beta^{-m_1} + \frac{1-m_2}{r}z_\beta^{1-m_1}\bigg)h_2^{\gamma(m_1-1)},
\end{aligned}
$$
and $z_\alpha$, $z_\beta$ are constants satisfying \eqref{eq: z_alpha_beta}.
We then prove $z_\alpha\in(0, 1-\alpha\delta]$ and $z_\beta \in [1+\beta\delta, +\infty)$.
Let
$$
\begin{aligned}
\phi_\alpha(z) &:= \frac{2}{\kappa^2 m_1(m_1-1)}z^{m_1} + \frac{z}{r}(m_2-1) + m_2\bigg(\alpha - \frac1\delta\bigg), \\
\phi_\beta(z) &:= \frac{2}{\kappa^2 m_2(m_2-1)}z^{m_2} + \frac{z}{r}(m_1-1) - m_1\bigg(\beta + \frac1\delta\bigg).
\end{aligned}
$$
By \eqref{eq: z_alpha_beta}, we have $\phi_\alpha(z_\alpha) = 0$.
Due to the fact that $m_1>1$, $m_2<0$, $m_1+m_2 = 1 + \frac{2}{\kappa^2}(r-\delta)$, $m_1m_2 = -\frac{2\delta}{\kappa^2}$ and assumption \ref{assume: SA}, we easily obtain $\phi_\alpha'(z) = \frac{2}{\kappa^2(m_1-1)}\big(z^{m_1-1}-1\big)< 0$ in $(0,1)$, $\phi_\alpha(0) = m_2\big(\alpha-\frac1\delta\big) > 0$, and $\phi_\alpha(1-\alpha\delta) \leq (1-\alpha\delta)m_2\alpha \leq 0$.
Thus, we conclude that $z_\alpha \in (0, 1-\alpha\delta]$.
Similarly, by \eqref{eq: z_alpha_beta}, we have $\phi_\beta(z_\beta) = 0$, $\phi_\beta'(z) = \frac{2}{\kappa^2(m_2-1)}\big(z^{m_2-1} - 1\big) > 0$ in $(1,+\infty)$, $\lim\limits_{z\rightarrow+\infty}\phi_\beta(z) = +\infty$, and $\phi_\beta(1+\beta\delta) \leq -(1+\beta\delta)m_1\beta \leq 0$.
Thus, we conclude that $z_\beta \in [1+\beta\delta, +\infty)$.

Therefore, it holds that
$$
\begin{aligned}
C_5(h_1, h_2) =& \int_0^{h_2} C_{5, h_2}(h_1, z)dz\\
=
&\frac{1-\gamma^*}{(m_1-m_2)(m_1-\gamma^*)}\bigg(\frac{m_2(\beta\delta+1)}{\delta}z_\beta^{-m_1} + \frac{1-m_2}{r}z_\beta^{1-m_1}\bigg)h_2^{1 + \gamma(m_1-1)},
\end{aligned}
$$
and
$$
\begin{aligned}
C_2(h_1, h_2) =& -\int_{h_1}^\infty C_{2, h_1}(z, h_2)dz \\
=&
\frac{1-\gamma^*}{(m_1-m_2)(m_2-\gamma^*)}\bigg(\frac{m_1(\alpha\delta-1)}{\delta}z_\alpha^{-m_2} + \frac{m_1-1}{r}z_\alpha^{1-m_2}\bigg) h_1^{1+\gamma(m_2-1)}.
\end{aligned}
$$
Together with \eqref{eq: C_diff}, we can obtain $C_1(h_1,h_2), \cdots, C_6(h_1,h_2)$ as in \eqref{eq: C_solu}.
Note that $v_{h_1}(y,h_1,h_2) = \alpha h_1^{-\gamma}$ for $y < z_\alpha h_1^{-\gamma}$, $v_{h_2}(y,h_1,h_2) = \-\beta h_2^{-\gamma}$ for $y > z_\beta h_2^{-\gamma}$, the dual value function has the explicit form in \eqref{eq: solu_dual_value}.
\end{proof}

\subsection{Proof of Theorem \ref{thm: verification} and Corollary \ref{cor: main_res}}\label{sec: proof_verification}
For any admissible strategy $(c,\pi) \in \mathcal{A}(x, h_1, h_2)$, we have the budget constraint
$\E\bigg[\int_0^\infty c_t \xi_t dt\bigg] \leq x$. We distinguish the following two pairs of reference processes, namely $H_{1,t}:= h_1 \vee \sup_{s\leq t}c_s$, $H_{2,t}:= h_2 \wedge \inf_{s\leq t}c_s$, and $H_{1,t}^\dag(y) = h_1\vee \sup_{s\leq t} c^\dag\big(Y_s(y), H_{1,s}^\dag(y), H_{2,s}^\dag(y)\big)$,  $H_{2,t}^\dag(y) = h_2\wedge \inf_{s\leq t} c^\dag\big(Y_s(y), $ $H_{1,s}^\dag(y), H_{2,s}^\dag(y)\big)$.
The two pairs of references are under an arbitrary consumption process $c_t$ and under the feedback consumption process $c^\dag$ with an arbitrary $y > 0$.
Note that the global optimal reference processes will be defined later by $H_{1,t}^* = H_{1,t}^\dag(y^*)$ and $H_{2,t}^* = H_{2,t}^\dag(y^*)$ with $y^*>0$ to be determined.
Let us now further introduce
\begin{equation}\label{eq: H_hat}
\begin{aligned}
\hH_{1,t}(y) &:= h_1 \vee \bigg(z_\alpha^{\frac1\gamma} \big(\inf\limits_{s\leq t} Y_s(y)\big)^{-\frac1\gamma}\bigg),\\
\hH_{2,t}(y) &:= h_2 \wedge \bigg(z_\beta^{\frac1\gamma} \big(\sup\limits_{s\leq t} Y_s(y)\big)^{-\frac1\gamma}\bigg).
\end{aligned}
\end{equation}

For any admissible controls $(c, \pi)\in\mathcal{A}(x, h_1, h_2)$ with reference processes $H_{1,t} = h_1\vee \sup\limits_{s\leq t}c_s$ and $H_{2,t} = h_2\wedge \inf\limits_{s\leq t}c_s$,
for any $y >0$,

\begin{equation}\label{eq: verification_main}
\begin{aligned}
&\E\bigg[\int_0^\infty e^{-\delta t}\big(U(c_t)dt -\alpha dV_{1t}^+ - \beta dV_{2t}^- \big) \bigg]\\
=& \E\bigg[\int_0^\infty e^{-\delta t}\big(U(c_t) - Y_t(y)c_t\big)dt -\alpha dV_{1t}^+ - \beta dV_{2t}^- \big) \bigg] + y\E\bigg[\int_0^\infty c_t \xi_t dt\bigg] \\
\leq& \E\bigg[\int_0^\infty e^{-\delta t}\big(U(c_t) - Y_t(y)c_t\big)dt -\alpha dV_{1t}^+ - \beta dV_{2t}^- \big) \bigg] + yx \\
\leq & v(y, h_1, h_2) + yx.
\end{aligned}
\end{equation}

According to Lemma \ref{lemma: dual_value_def} and Lemma \ref{lemma: y_opt}, all the inequalities becomes equalities with the choices $c_t^* = c^\dag(Y_t(y^*), \hH_{1,t}(y^*), \hH_{2,t}(y^*))$, in which $y^*$ is the solution to the equation
\begin{equation}\label{eq: budget_constraint_eq}
\E\bigg[\int_0^\infty c^\dag(Y_t(y^*), \hH_{1,t}(y^*), \hH_{2,t}(y^*))\xi_t dt \bigg] = x,
\end{equation}
for any $x\geq 0$.
In conclusion, we have
$$
\sup\limits_{(c,\pi)\in\mathcal{A}(x,h_1,h_2)} \E\bigg[\int_0^\infty e^{-\delta t}\big(U(c_t)dt -\alpha dV_{1t}^+ - \beta dV_{2t}^- \big) \bigg] = \inf\limits_{y>0}\big(v(y, h_1, h_2) + yx\big) = u(x, h_1, h_2),
$$
which completes the proof of the verification theorem.

Then we propose the lemmas used in the proof of the Theorem \ref{thm: verification}.

We first show the asymptotic results for the coefficients defined in Proposition \ref{prop: solu_dual_value}, which are used for the proof of lemmas to be proposed.
\begin{remark}\label{rmk: order_coef}
Based on the explicit formulas in \eqref{eq: C_solu}, we note that as $h_2\rightarrow0$, $h_1\rightarrow+\infty$, we have the asymptotic properties
$$
\begin{aligned}
&C_1(h_1, h_2) = O\big(h_1^{1+\gamma(m_1-1)}\big), ~
C_2(h_1) = O\big(h_1^{1+\gamma(m_2-1)}\big), ~
C_3(h_2) = O\big(h_2^{1+\gamma(m_1-1)}\big) \\
&C_4(h_1) = O\big(h_1^{1+\gamma(m_2-1)}\big),~
C_5(h_2) = O\big(h_2^{1+\gamma(m_1-1)}\big),~
C_6(h_1, h_2) = O\big(h_2^{1+\gamma(m_2-1)}\big).
\end{aligned}
$$
\end{remark}

\begin{lemma}\label{lemma: dual_value_def}
The dual value function is indeed
$$
\begin{aligned}
v(y, h_1, h_2)
&= \sup\limits_{(c, \cdot)\in\mathcal{A}(h_1,h_2)} \E\bigg[\int_0^\infty e^{-\delta t}\big((U(c_t) -Y_t(y)c_t)\big)dt - \alpha V'(H_{1,t})dH_{1,t} + \beta V'(H_{2,t})dH_{2,t}\bigg] \\
&= \sup\limits_{(c, \cdot)\in\mathcal{A}(h_1,h_2)} \E\bigg[\int_0^\infty e^{-\delta t}\big((U(c_t) -Y_t(y)c_t)\big)dt - \alpha dV_{1,t}^+ - \beta dV_{2,t}^-\bigg],
\end{aligned}
$$
and the optimal consumption path $c_t = c^\dag(Y_t(y), H_{1,t}^\dag(y), H_{2,t}^\dag(y))$.
\end{lemma}
\begin{proof}
Recall that $Y_t(y) = ye^{\delta t}\xi_t = ye^{\big(\delta-r-\frac{\kappa^2}{2}\big) t - \kappa W_t}$, indicating that
$$
dY_t(y) = Y_t(y)\big((\delta-r)dt - \kappa dW_t\big).
$$

For a given admissible consumption process $\{c_t\}_{t=0}^\infty$ with reference processes $H_{1,t} = h_1\vee \sup\limits_{s\leq t}c_s$ and $H_{2,t} = h_2\wedge \inf\limits_{s\leq t}c_s$,
we define $M_t$ by
$$
M_t^c = \int_0^t e^{-\delta s}\big( (U(c_s)-Y_s(y)c_s)ds - \alpha V'(H_{1,s})dH_{1,s} + \beta V'(H_{2,s})dH_{2,s} \big) + e^{-\delta t} v(Y_t(y), H_{1,t}, H_{2,t}).
$$
By It$\hat{\mathrm{o}}$'s formula, we can derive that
\begin{equation}\label{eq: dual_value_process_deri}
\begin{aligned}
dM_t^c =& e^{-\delta t}\big((U(c_t)-Y_t(y)c_t)dt - \alpha V'(H_{1,t})dH_{1,t} + \beta V'(H_{2,t})dH_{2,t} \\
&+ dv(Y_t(y), H_{1,t}, H_{2,t}) - \delta v(Y_t(y), H_{1,t}, H_{2,t})dt\big)  \\
\leq& e^{-\delta t}\big( -\delta v\big(Y_t(y), H_{1,t}, H_{2, t}\big) + (\delta-r)Y_t(y)v_y\big(Y_t(y), H_{1,t}, H_{2, t}\big) \\
& + \frac{\kappa^2}{2}Y_t(y)^2 v_{yy}\big(Y_t(y), H_{1,t}, H_{2, t} \big) + \tU(Y_t(y), H_{1,t}, H_{2,t})\big)dt\\
&- \kappa e^{-\delta t}Y_t(y)v_y\big(Y_t(y), H_{1,t}, H_{2, t}\big) dW_t\\
& + e^{-\delta t} \big(v_{h_1}\big(Y_t(y), H_{1,t}, H_{2, t}\big) - \alpha V'(H_{1,t})\big)dH_{1,t}\\
& + e^{-\delta t}\big(v_{h_2}\big(Y_t(y), H_{1,t}, H_{2, t}\big) + \beta V'(H_{2,t})\big)dH_{2,t}.\\
\end{aligned}
\end{equation}
Here we recall that $\tilde{U}(q,h_1,h_2)=\sup_{c\in[h_2,h_1]}(U(c)-cq)$ for $q\geq 0,h_1\geq h_2\geq 0$.

Then, for any fixed $T>0$, we have
\begin{equation}\label{eq: M_t_inte}
\begin{aligned}
M_T^c - M_t^c
\leq& \int_t^T e^{-\delta t}\big( -\delta v\big(Y_s(y), H_{1,s}, H_{2, s}\big) + (\delta-r)Y_s(y)v_y\big(Y_s(y), H_{1,s}, H_{2, s}\big) \\
& + \frac{\kappa^2}{2}Y_s(y)^2 v_{yy}\big(Y_s(y), H_{1,s}, H_{2, s} \big) + \tU(Y_s(y), H_{1,s}, H_{2,s})\big)ds \\
& +\int_t^T (- \kappa) e^{-\delta s}Y_s(y)v_y\big(Y_s(y), H_{1,s}, H_{2, s}\big) dW_s \\
& + \int_t^T e^{-\delta t} \big(v_{h_1}\big(Y_s(y), H_{1,s}, H_{2, s}\big) - \alpha V'(H_{1,s})\big)dH_{1,s} \\
& + \int_t^T e^{-\delta t}\big(v_{h_2}\big(Y_s(y), H_{1,s}, H_{2,s}\big) + \beta V'(H_{2,s})\big)dH_{2,s} \\
:= & I_1 + I_2 + I_3 + I_4.
\end{aligned}
\end{equation}
Here, we have $\E_t[I_1]\leq 0$, $\E_t[I_3]\leq 0$ and $\E_t[I_4]\leq 0$ because of \eqref{eq: HJB_dual_ineq}, and $\E_t[I_2] = 0$ due to Lemma \ref{lemma: N_martingale}.
Therefore, we conclude that $\E_t[M_T^c - M_t^c] \leq 0$ and $\{M_t^c\}$ is a super-martingale.
This implies that for any $T>0$,
\begin{align}\label{eq:super-martingale-v}
v(y, h_1, h_2) =& M_0^c \geq \E[M_T^c] \nonumber\\
=& \E\bigg[\int_0^T e^{-\delta t}\big( (U(c_t)-Y_t(y)c_t)dt - \alpha V'(H_{1,t})dH_{1,t} + \beta V'(H_{2,t})dH_{2,t} \big) \bigg]\nonumber\\
&+ \E\bigg[ e^{-\delta T} v(Y_T(y), H_{1,T}, H_{2,T})\bigg].
\end{align}
By letting $T\to \infty$ in \eqref{eq:super-martingale-v}, we obtain from Lemma \ref{lemma: transversality}, Monotone Convergence Theorem and Dominated Convergence Theorem that
$$
\begin{aligned}
v(y, h_1, h_2) \geq& \E\bigg[\int_0^{\infty} e^{-\delta t}\big( (U(c_t)-Y_t(y)c_t)dt - \alpha V'(H_{1,t})dH_{1,t} + \beta V'(H_{2,t})dH_{2,t} \big) \bigg].
\end{aligned}
$$
The above inequality holds for any feasible consumption policy $c$, and we obtain
$$
\begin{aligned}
v(y, h_1, h_2) \geq \sup\limits_{(c, \cdot)\in\mathcal{A}(h_1,h_2)} \E\bigg[\int_0^{\infty} e^{-\delta t}\big( (U(c_t)-Y_t(y)c_t)dt - \alpha V'(H_{1,t})dH_{1,t} + \beta V'(H_{2,t})dH_{2,t} \big) \bigg] \\
\geq \sup\limits_{(c, \cdot)\in\mathcal{A}(h_1,h_2)} \E\bigg[\int_0^{\infty} e^{-\delta t}\tU(y, H_{1,t}, H_{2,t})dt - \alpha V'(H_{1,t})dH_{1,t} + \beta V'(H_{2,t})dH_{2,t} \big) \bigg].
\end{aligned}
$$
It is easy to verify that when $c_t = c^\dag(Y_t(y), H_{1,t}^\dag(y), H_{2,t}^\dag(y))$, $\{M_t^c\}$ becomes a martingale by using Proposition \ref{prop: solu_dual_value} and Lemma \ref{lemma: transversality}, then
all the inequalities become equalities. Thus, we get the desired result.
\end{proof}

\begin{lemma}\label{lemma: N_martingale}
For a given admissible consumption process $\{c_t\}_{t=0}^\infty$ with reference processes $H_{1,t} = h_1\vee \sup\limits_{s\leq t}c_s$ and $H_{2,t} = h_2\wedge \inf\limits_{s\leq t}c_s$, let us define
$$
N_t := \int_0^t -\kappa e^{-\delta s}Y_s(y)v_y\big(Y_s(y), H_{1,s}, H_{2, s}\big) dW_s.
$$
Then $N_t$ is a martingale.
\end{lemma}
\begin{proof}
In order to show that $N_t$ is a martingale, it is sufficient to show
$$
E\bigg[\int_0^T \bigg(e^{-\delta t}\big(-\kappa Y_t(y)\big)v_y(Y_t(y), H_{1,t}, H_{2,t})\bigg)^2dt \bigg] < \infty, ~\forall ~T\geq 0.
$$
According to the proof of Lemma \ref{lemma: v_convex},
if $Y_t(y) < z_\alpha H_{1,t}^{-\gamma}$, $v_y(Y_t(y), H_{1,t}, H_{2,t}) = v_y(Y_t(y), \bh_1(Y_t(y)), H_{2,t})$ and $Y_t(y) = z_\alpha \bh_1(Y_t(y))^{-\gamma}$;
if $Y_t(y) > z_\beta H_{2,t}^{-\gamma}$, $v_y(Y_t(y), H_{1,t}, H_{2,t}) = v_y(Y_t(y), H_{1,t}, \uh_2(Y_t(y)))$ and $Y_t(y) = z_\beta \uh_2(Y_t(y))^{-\gamma}$.
Therefore, it is sufficient to discuss the cases in region $z_\alpha H_{1,t}^{-\gamma} \leq Y_t(y) \leq z_\beta H_{2,t}^{-\gamma}$ according to Proposition \ref{prop: solu_dual_value} and Remark \ref{rmk: order_coef}.\\

\noindent
(i) If $z_\alpha H_{1,T}^{-\gamma} \leq Y_T(y) < H_{1,T}^{-\gamma}$, we have
$$
\begin{aligned}
&v_y\big(Y_T(y), H_{1,T}, H_{2, T}\big)
= m_1C_1(H_{1,T}, H_{2,T})Y_T(y)^{m_1-1} + m_2C_2(H_{1,T})Y_T(y)^{m_2-1}  - \frac{H_{1,T}}{r}\\
= &O\bigg(H_{1,T}^{1+\gamma(m_1-1)}Y_T(y)^{m_1-1}\bigg) + O\bigg(H_{1,T}^{1+\gamma(m_2-1)}Y_T(y)^{m_2-1}\bigg) + O\bigg(Y_T(y)^{\gamma^*-1}\bigg)\\
=& O\bigg(Y_T(y)^{\gamma^*-1}\bigg),
\end{aligned}
$$
where the last equation holds using the fact that $H_{1,T} = O\big(Y_T(y)^{-1/\gamma}\big)$ and $H_{1,T}^{-1} = O\big(Y_T(y)^{1/\gamma}\big)$ in this region.\\

\noindent
(ii) If $H_{1,T}^{-\gamma} \leq Y_T(y) \leq H_{2,T}^{-\gamma}$, we have
$$
\begin{aligned}
&v_y\big(Y_T(y), H_{1,T}, H_{2, T}\big)\\
=&m_1C_3( H_{2,T})Y_T(y)^{m_1-1} + m_2C_4(H_{1,T})Y_T(y)^{m_2-1} + \frac{2}{\kappa^2 (\gamma^*-m_1)(\gamma^*-m_2)}Y_T(y)^{\gamma^*-1} \\
=& O\bigg(H_{2,T}^{1+\gamma(m_1-1)}Y_T(y)^{m_1-1}\bigg) + O\bigg(H_{1,T}^{1+\gamma(m_2-1)}Y_T(y)^{m_2-1}\bigg) + O\bigg(Y_T(y)^{\gamma^*-1}\bigg) \\
=& O\bigg(Y_T(y)^{\gamma^*-1}\bigg),
\end{aligned}
$$
where the last equation holds using the fact that $H_{2,T} = O\big(Y_T(y)^{-1/\gamma}\big)$ and $H_{1,T}^{-1} = O\big(Y_T(y)^{1/\gamma}\big)$ in this region.\\

\noindent
(iii) If $H_{2,T}^{-\gamma} \leq Y_T(y) < z_\beta H_{2,T}^{-\gamma}$, we have
$$
\begin{aligned}
&v_y\big(Y_T(y), H_{1,T}, H_{2, T}\big) = C_5( H_{2,T})Y_T(y)^{m_1-1} + C_6(H_{1,T}, H_{2,T})Y_T(y)^{m_2-1} - \frac{\hH_{2,T}(y)}{r}\\
= & O\bigg(H_{2,T}^{1+\gamma(m_1-1)}Y_T(y)^{m_1-1}\bigg) + O\bigg(H_{2,T}^{1+\gamma(m_2-1)}Y_T(y)^{m_2-1}\bigg) + O\bigg(Y_T(y)^{\gamma^*-1}\bigg) \\
= & O\bigg(Y_T(y)^{\gamma^*-1}\bigg),
\end{aligned}
$$
where the last equation holds using the fact that $\hH_{2,T}(y) = O\big(Y_T(y)^{-1/\gamma}\big)$ and $\hH_{2,T}(y)^{-1} = O\big(Y_T(y)^{1/\gamma}\big)$ in this region.\\
In summary, we have $v_y(Y_t(y), H_{1,T}, H_{2,T}) = O(Y_t(y)^{\gamma^*-1})$.
Therefore, we deduce that
$$
\begin{aligned}
&\bigg(e^{-\delta t}\big(-\kappa Y_t(y)\big)v_y(Y_t(y), H_{1,t}, H_{2,t})\bigg)^2
= O\bigg(e^{-2\delta t} Y_t(y)^{2\gamma^*}\bigg),
\end{aligned}
$$
and thus
\begin{align*}
&\E\bigg[\int_0^T \bigg(e^{-\delta t}\big(-\kappa Y_t(y)\big)v_y(Y_t(y), H_{1,t}, H_{2,t})\bigg)^2  dt\bigg] \\
\leq &K \cdot \E\bigg[\int_0^T e^{-2\delta t} Y_t(y)^{2\gamma^*}dt\bigg] = K\cdot \exp\{-K_0 T\} < \infty,
\end{align*}
where $K_0 > 0$ is defined in Assumption \ref{assume: exist}.
\end{proof}
\begin{lemma}\label{lemma: transversality}
For a given admissible consumption process $\{c_t\}_{t=0}^\infty$ with reference processes $H_{1,t} = h_1\vee \sup\limits_{s\leq t}c_s$ and $H_{2,t} = h_2\wedge \inf\limits_{s\leq t}c_s$, we have the transversality condition
$$
\lim\limits_{T\rightarrow+\infty}\E\bigg[e^{-\delta T} v\big(Y_{T}(y), H_{1, T}, H_{2, T}\big)\bigg] = 0.
$$
\end{lemma}
\begin{proof}
Let us consider the following cases respectively. 

\noindent
(i) If $z_\alpha H_{1,T}^{-\gamma} \leq Y_T(y) < H_{1,T}^{-\gamma}$, we have
\begin{align*}
v\big(Y_T(y), H_{1,T}, H_{2, T}\big)
= &C_1(H_{1,T}, H_{2,T})Y_T(y)^{m_1} + C_2(H_{1,T})Y_T(y)^{m_2} + \frac{H_{1,T}^{1-\gamma}}{\delta(1-\gamma)} - \frac{H_{1,T}Y_T(y)}{r}\\
= &O\bigg(H_{1,T}^{1+\gamma(m_1-1)}Y_T(y)^{m_1}\bigg) + O\bigg(H_{1,T}^{1+\gamma(m_2-1)}Y_T(y)^{m_2}\bigg) + O\bigg(Y_T(y)^{\gamma^*}\bigg)\\
=& O\bigg(Y_T(y)^{\gamma^*}\bigg),
\end{align*}
where the last equation holds using the fact that $H_{1,T} = O\big(Y_T(y)^{-1/\gamma}\big)$ and $H_{1,T}^{-1} = O\big(Y_T(y)^{1/\gamma}\big)$ in this region.\\

\noindent
(ii) If $H_{1,T}^{-\gamma} \leq Y_T(y) \leq H_{2,T}^{-\gamma}$, we have
$$
\begin{aligned}
v\big(Y_T(y), H_{1,T}, H_{2, T}\big)
=&C_3( H_{2,T})Y_T(y)^{m_1} + C_4(H_{1,T})Y_T(y)^{m_2} +
\frac{2}{\kappa^2 \gamma^*(\gamma^*-m_1)(\gamma^*-m_2)}Y_T(y)^{\gamma^*} \\
=& O\bigg(H_{2,T}^{1+\gamma(m_1-1)}Y_T(y)^{m_1}\bigg) + O\bigg(H_{1,T}^{1+\gamma(m_2-1)}Y_T(y)^{m_2}\bigg) + O\bigg(Y_T(y)^{\gamma^*}\bigg)\\
=& O\bigg(Y_T(y)^{\gamma^*}\bigg),
\end{aligned}
$$
where the last equation holds using the fact that $H_{2,T} = O\big(Y_T(y)^{-1/\gamma}\big)$ and $H_{1,T}^{-1} = O\big(Y_T(y)^{1/\gamma}\big)$ in this region.\\

\noindent
(iii) If $H_{2,T}^{-\gamma} \leq Y_T(y) < z_\beta H_{2,T}^{-\gamma}$, we have
$$
\begin{aligned}
v\big(Y_T(y), H_{1,T}, H_{2, T}\big)
= &C_5( H_{2,T})Y_T(y)^{m_1} + C_6(H_{1,T}, H_{2,T})Y_T(y)^{m_2}
+ \frac{H_{2,T}^{1-\gamma}}{\delta(1-\gamma)} - \frac{H_{2,T}Y_t(y)}{r}\\
= & O\bigg(H_{2,T}^{1+\gamma(m_1-1)}Y_T(y)^{m_1}\bigg) + O\bigg(H_{2,T}^{1+\gamma(m_2-1)}Y_T(y)^{m_2}\bigg) + O\bigg(Y_T(y)^{\gamma^*}\bigg) \\
= & O\bigg(Y_T(y)^{\gamma^*}\bigg),
\end{aligned}
$$
where the last equation holds using the fact that $H_{2,T}(y) = O\big(Y_T(y)^{-1/\gamma}\big)$ and $H_{2,T}(y)^{-1} = O\big(Y_T(y)^{1/\gamma}\big)$ in this region.\\

\noindent
(iv) If $Y_T(y) < z_\alpha H_{1,T}^{-\gamma}$, we have
$$
\begin{aligned}
v\big(Y_T(y), H_{1,T}, H_{2, T}\big)
=& v\big(Y_T(y), \bh_1(Y_T(y)), H_{2, T}\big) + \frac{\alpha}{1-\gamma}\big(H_{1,T}^{1-\gamma} - \bh_1(Y_T(y))^{1-\gamma}\big) \\
=& O(1) + O\bigg(Y_T(y)^{\gamma^*}\bigg).
\end{aligned}
$$

\noindent
(v) If $Y_T(y) > z_\beta H_{2,T}^{-\gamma}$, we similarly have
$$
v\big(Y_T(y), H_{1,T}, H_{2, T}\big) = O(1) + O\bigg(Y_T(y)^{\gamma^*}\bigg).
$$

For any $m \in (m_2, m_1)$, it is straightforward to check that
\begin{align}\label{limit-Y}
\lim\limits_{T\rightarrow+\infty} \E\big[e^{-\delta T}Y_T(y)^{m}\big] = 0.
\end{align}

In summary, we deduce that $v\big(Y_T(y), H_{1,T}, H_{2, T}\big) = O(1) + O\big(Y_T(y)^{\gamma^*}\big)$, and thus
$$
\mathbb{E}\left[e^{-\delta T}v\big(Y_T(y), H_{1,T}, H_{2, T}\big)\right] \rightarrow0
$$
as $T\rightarrow0$ in view of \eqref{limit-Y}.
\end{proof}

The proof of the next result is similar to the one of Lemma 5.4 in \cite{DengLiPY2022FS}, which is thus omitted. 

\begin{lemma}\label{lemma: y_opt}
The first inequality in \eqref{eq: verification_main} becomes equality with $c_t = c^\dag(Y_t(y^*), \hH_{1,t}(y^*), \hH_{2,t}(y^*))$, with $y^* = y^*(x,h)$ as the unique solution to \eqref{eq: budget_constraint_eq}.
\end{lemma}

To prove Corollary \ref{cor: main_res}, it is sufficient to prove the existence of a unique strong solution to SDE \eqref{eq: SDE_opt} under the optimal controls.
First, we need to establish some results concerning the regularity of the feedback controls $c^*(x, h_1, h_2)$ and $\pi^*(x, h_1, h_2)$. The proof of the next result is similar to the proof of Lemma 5.7 in \cite{DengLiPY2022FS}

\begin{lemma}\label{lemma: f_regularity}
The function $f(\cdot, h_1,h_2)$ is $C^1$ within all five subsets of $\{(y, h_1, h_2)\in\mathbb{R}_+^3: h_1\geq h_2\}$, and is continuous at points $x = x_\mathrm{lavs}(h_1,h_2)$, $x=x_\mathrm{peak}(h_1,h_2)$, $x = x_\mathrm{valy}(h_1,h_2)$, and $x = x_\mathrm{gloom}(h_1,h_2)$.
Moreover, we have
$$
\begin{aligned}
f_x(x, h_1, h_2) &= \frac{1}{g_y(f(x, h_1, h_2), h_1, h_2)} \\
%&=
%\begin{cases}
%\begin{aligned}
%\bigg(&-m_1(m_1-1)C_1(h_1,h_2)f(x,h_1,h_2)^{m_1-2} \\
%&- m_2(m_2-1)C_2(h_1,h_2)f(x,h_1,h_2)^{m_2-2}\bigg)^{-1},
%\end{aligned}
%&\mbox{if } x_\mathrm{peak}(h_1,h_2) \leq x \leq x_\mathrm{lavs}(h_1,h_2), \\
%\begin{aligned}
%\bigg(&-m_1(m_1-1)C_3(h_1,h_2)f(x,h_1,h_2)^{m_1-2} \\
%&- m_2(m_2-1)C_4(h_1,h_2)f(x,h_1,h_2)^{m_2-2}\\
%&- \frac{2(\gamma^*-1)}{\kappa^2(\gamma^*-m_1)(\gamma^*-m_2)}f(x,h_1,h_2)^{\gamma^*-2}\bigg)^{-1},
%\end{aligned}
%&\mbox{if } x_\mathrm{valy}(h_1,h_2) \leq x \leq x_{peak}(h_1,h_2), \\
%\begin{aligned}
%\bigg(&-m_1(m_1-1)C_5(h_1,h_2)f(x,h_1,h_2)^{m_1-2} \\
%&- m_2(m_2-1)C_6(h_1,h_2)f(x,h_1,h_2)^{m_2-2}\bigg)^{-1},
%\end{aligned}
%&\mbox{if } x_\mathrm{gloom}(h_1,h_2) \leq x \leq x_\mathrm{valy}(h_1,h_2),  \\
%f_x(x, \bh_1(f(x, h_1, h_2)), h_2), &\mbox{if } x > x_\mathrm{lavs}(h_1,h_2), \\
%f_x(x, h_1, \uh_2(f(x, h_1, h_2))), &\mbox{if } x < x_\mathrm{gloom}(h_1, h_2).
%\end{cases}
\end{aligned}
$$
and
$$
f_{h_i}(x, h_1, h_2) = -g_{h_i}\big(f(x, h_1, h_2), h_1, h_2\big)\cdot {f_x(x, h_1, h_2)},~~for~i=1,2.
$$
\end{lemma}

\begin{lemma}\label{lemma: control_Lipschitz}
The function $c^*(x, h_1, h_2)$ is locally Lipschitz, and $\pi^*(x, h_1, h_2)$ is Lipschitz when $h_1\geq h_2$.
\end{lemma}
\begin{proof}
See Section \ref{sec: proof_control_Lipschitz}.
\end{proof}

Based on the previous results, the next proposition holds.  
\begin{proposition}\label{prop: SDE_solu_exist}
SDE \eqref{eq: SDE_opt} has a unique strong solution $(X^*, H_1^*, H_2^*)$ for any initial condition $(x,h_1,h_2)\in\mathbb{R}_+^3$ such that $h_1\geq h_2$.
\end{proposition}

\subsection{Proof of Lemma \ref{lemma: v_convex}}\label{sec: proof_lemma_v_convex}
We aim to show the result in different regions respectively.

%Therefore, we then consider $v_{yy}(y,h_1,h_2)$ in region $z_\alpha h_1^{-\gamma} \leq y \leq z_\beta h_2^{-\gamma}$. \\
%
\begin{itemize}
    \item[(i)]In region $z_\alpha h_1^{-\gamma} \leq y < h_1^{-\gamma}$. We have
\end{itemize}
$$
\begin{aligned}
v_{yy}(y,h_1,h_2) &= m_1(m_1-1)C_1(h_1,h_2)y^{m_1-2} + m_2(m_2-1)C_2(h_1)y^{m_2-2}. \\
%&= y^{m_2-2}\bigg(m_1(m_1-1)C_1(h_1,h_2)y^{m_1-m_2} + m_2(m_2-1)C_2(h_1,h_2)\bigg).
\end{aligned}
$$
It is sufficient to show $C_1(h_1,h_2) > 0$ and $C_2(h_1) > 0$.
Because $m_1 > 1$, $m_2<0$, $m_1>\gamma^*>m_2$, and $\gamma^*<1$, we first have
$$
\begin{aligned}
C_2(h_1,h_2)
%&= \frac{1-\gamma^*}{(m_1-m_2)(m_2-\gamma^*)}\bigg(\frac{m_1(\alpha\delta-1)}{\delta}z_\alpha^{-m_2} + \frac{m_1-1}{r}z_\alpha^{1-m_2}\bigg) h_1^{1+\gamma(m_2-1)}\\
&= \frac{1-\gamma^*}{(m_1-m_2)(m_2-\gamma^*)}\bigg(\frac{m_1(\alpha\delta-1)}{\delta} + \frac{m_1-1}{r}z_\alpha\bigg) z_\alpha^{-m_2}h_1^{1+\gamma(m_2-1)}\\
&\geq \frac{1-\gamma^*}{(m_1-m_2)(m_2-\gamma^*)}\bigg(\frac{m_1(\alpha\delta-1)}{\delta} + \frac{m_1-1}{r}(1-\alpha\delta)\bigg) z_\alpha^{-m_2}h_1^{1+\gamma(m_2-1)}\\
&=\frac{2(1-\gamma^*)(1-\alpha\delta)}{\kappa^2(m_1-m_2)(\gamma^*-m_2)m_2(m_2-1)} z_\alpha^{-m_2}h_1^{1+\gamma(m_2-1)} > 0,
\end{aligned}
$$
where the first inequality holds because $z_\alpha\in(0, 1-\alpha\delta]$.
Then we consider $C_1(h_1,h_2)$.
By similar argument for $C_2(h_1,h_2)$, we can also deduce that $C_5(h_1,h_2) > 0$.
Therefore,
$$
C_1(h_1,h_2) = C_5(h_2) + \frac{2(1-\gamma^*)}{\kappa^2(m_1-m_2)m_1(m_1-1)(m_1-\gamma^*)}\big(h_1^{1+\gamma(m_1-1)} - h_2^{1+\gamma(m_1-1)}\big) > 0.
$$
\begin{itemize}
\item[(ii)]In region $h_2^{-\gamma} < y \leq z_\beta h_2^{-\gamma}$. Note that
\end{itemize}
$$
\begin{aligned}
v_{yy}(y,h_1,h_2) &= m_1(m_1-1)C_5(h_2)y^{m_1-2} + m_2(m_2-1)C_6(h_1,h_2)y^{m_2-2}.
\end{aligned}
$$
It is sufficient to show $C_5(h_1,h_2) > 0$ and $C_6(h_1) > 0$.
As our discussion above, $C_5(h_2) > 0$.
Then we consider $C_6(h_1,h_2)$.
Indeed,
$$
C_6(h_1,h_2) = C_2(h_1) + \frac{2(1-\gamma^*)}{\kappa^2(m_1-m_2)m_2(m_2-1)(m_2-\gamma^*)}\big(h_1^{1+\gamma(m_2-1)} - h_2^{1+\gamma(m_1-1)}\big) > 0.
$$
\begin{itemize}
\item[(iii)]In region $h_1^{-\gamma}\leq y \leq h_2^{-\gamma}$. By $C_5(h_1,h_2) > 0$ and $C_2(h_1) > 0$, we obtain
\end{itemize}
$$
\begin{aligned}
&v_{yy}(y,h_1,h_2) \\
%&\quad=m_1(m_1-1)C_3(h_1,h_2)y^{m_1-2} + m_2(m_2-1)C_4(h_1,h_2)y^{m_2-2} + \frac{2(\gamma^*-1)}{\kappa^2(\gamma^*-m_1)(\gamma^*-m_2)}y^{\gamma^*-2} \\
&\quad= \bigg(m_1(m_1-1)C_3(h_2)y^{m_1-\gamma^*} + m_2(m_2-1)C_4(h_1)y^{m_2-\gamma^*} \\
&\qquad+ \frac{2(\gamma^*-1)}{\kappa^2(\gamma^*-m_1)(\gamma^*-m_2)}\bigg)y^{\gamma^*-2} \\
&\quad> \bigg(m_1(m_1-1)(C_3(h_2)-C_5(h_2))y^{m_1-\gamma^*} + m_2(m_2-1)(C_4(h_1)-C_2(h_1))y^{m_2-\gamma^*} \\
&\qquad+ \frac{2(\gamma^*-1)}{\kappa^2(\gamma^*-m_1)(\gamma^*-m_2)}\bigg)y^{\gamma^*-2} \\
&\quad \geq \bigg(-\frac{2(1-\gamma^*)}{\kappa^2(m_1-m_2)(m_1-\gamma^*)} + \frac{2(1-\gamma^*)}{\kappa^2(m_1-m_2)(m_2-\gamma^*)}+ \frac{2(\gamma^*-1)}{\kappa^2(\gamma^*-m_1)(\gamma^*-m_2)}\bigg)y^{\gamma^*-2} \\
&\quad= 0.
\end{aligned}
$$
\begin{itemize}
\item[(iv)] In region $y < z_\alpha h_1^{-\gamma}$. It follows that
\end{itemize}
$$
\begin{aligned}
v_y(y, h_1, h_2) 
&= v_y\big(y, \bh_1(y), h_2\big) + v_{h_1}\big(y, \bh_1(y), h_2\big)\cdot\frac{\partial \bh_1(y)}{\partial y} - \alpha\bh_1(y)^{-\gamma}\frac{\partial \bh_1(y)}{\partial y}\\
&= v_y\big(y, \bh_1(y), h_2\big) + \big(v_{h_1}\big(y, \bh_1(y), h_2\big) - \alpha V'(\bh_1(y))\big)\cdot\frac{\partial \bh_1(y)}{\partial y} \\
&= v_y\big(y, \bh_1(y), h_2\big),
\end{aligned}
$$
where the last equation holds because of \eqref{eq: dual_boundary_adjust1}, and thus
$$
\begin{aligned}
v_{yy}(y, h_1, h_2) = \frac{\partial v_y\big(y, \bh_1(y), h_2\big)}{\partial y}
&= v_{yy}(y, \bh_1(y), h_2) - v_{yh_1}(y, \bh_1(y), h_2)\cdot\frac{\partial \bh_1(y)}{\partial y}\\
&= v_{yy}(y, \bh_1(y), h_2) > 0,
\end{aligned}
$$
where the last equation holds because of \eqref{eq: dual_boundary_adjust2}.
\begin{itemize}
    \item[(v)] In region $y > z_\beta h_2^{-\gamma}$. Similarly,  we also have $v_y(y, h_1, h_2) = v_y\big(y, h_1, \uh_2(y)\big)$, and thus $v_{yy}(y, h_1, h_2) > 0$.
\end{itemize}
Thus, we complete the proof of this lemma.

\subsection{Proof of Lemma \ref{lemma: control_Lipschitz}}\label{sec: proof_control_Lipschitz}
From the definition of $c^*(x,h_1,h_2)$ in \eqref{eq: solu_consume}, it is clear that $c^*$ is locally Lipschitz.

Next, we prove that $\pi^*(x,h_1,h_2)$ is Lipschitz.
From \eqref{eq: f_obtain}, for $ x_\mathrm{peak}(h_1,h_2) < x \leq x_\mathrm{lavs}(h_1, h_2)$, it follows that the function $f$ is determined by
\begin{eqnarray}
x=-m_1C_1(h_1,h_2)f(x,h_1,h_2)^{m_1-1}-m_2C_2(h_1)f(x,h_1,h_2)^{m_2-1}+\frac{h_1}{r}.\nonumber
\end{eqnarray}
Differentiating the above equation with respect to $x$, we have
\begin{eqnarray}
1=-\big(m_1(m_1-1)C_1(h_1,h_2)f(x,h_1,h_2)^{m_1-2}+m_2(m_2-1)C_2(h_1)f(x,h_1,h_2)^{m_2-2}\big)f_x(x,h_1,h_2).\nonumber
\end{eqnarray}
Plugging this equation into $\pi^*_x$ and using Lemma \ref{lemma: f_regularity}, we obtain
\begin{eqnarray}
\hspace{-0.3cm}&&\hspace{-0.3cm}
\pi^*_x(x,h_1,h_2)
\nonumber\\
\hspace{-0.3cm}&=&\hspace{-0.3cm}
\frac{\mu-r}{\sigma^2}\left(m_1(m_1-1)(m_1-m_2)C_1(h_1,h_2)f(x,h_1,h_2)^{m_1-2}f_x(x,h_1,h_2)+1-m_2\right)
\nonumber\\
%\hspace{-0.3cm}&=&\hspace{-0.3cm}
%\frac{\mu-r}{\sigma^2}\left(\frac{m_1(m_1-1)(m_1-m_2)C_1(h_1,h_2)f(x,h_1,h_2)^{m_1-2}}{-m_1(m_1-1)C_1(h_1,h_2)f(x,h_1,h_2)^{m_1-2} - m_2(m_2-1)C_2(h_1,h_2)f(x,h_1,h_2)^{m_2-2}}+1-m_2\right)
%\nonumber\\
%\hspace{-0.3cm}&=&\hspace{-0.3cm}
%\frac{\mu-r}{\sigma^2}\left(\frac{m_1(m_1-1)(m_1-m_2)C_1(h_1,h_2)}{-m_1(m_1-1)C_1(h_1,h_2)- m_2(m_2-1)C_2(h_1)f(x,h_1,h_2)^{m_2-m_1}}+1-m_2\right)
%\nonumber\\
\hspace{-0.3cm}&\leq&\hspace{-0.3cm}
\frac{\mu-r}{\sigma^2}\left(\frac{m_1(m_1-1)(m_1-m_2)C_1(h_1,h_2)}{-m_1(m_1-1)C_1(h_1,h_2)- m_2(m_2-1)C_2(h_1)(z_{\alpha}h_1^{-\gamma})^{m_2-m_1}}+1-m_2\right),\nonumber
\end{eqnarray}
where in the last inequality we have used the facts that the function $y\mapsto-m_1(m_1-1)C_1(h_1,h_2)- m_2(m_2-1)C_2(h_1)(z_{\alpha}y^{-\gamma})^{m_2-m_1}$ is strictly negative and decreasing, and $f(x,h_1,h_2)\in[z_{\alpha}h_1^{-\gamma},h_1^{-\gamma})$ for $x\in(x_\mathrm{peak}(h_1,h_2),x_\mathrm{lavs}(h_1, h_2)]$.
Furthermore, in view of functions $(C_i(h_1,h_2))_{i=1,...,6}$ in \eqref{eq: C_solu}, it is easy to verify that the above inequality can be rewritten as
\begin{eqnarray}
\pi^*_x(x,h_1,h_2)
\hspace{-0.3cm}&\leq&\hspace{-0.3cm}
\frac{\mu-r}{\sigma^2}\left(\frac{m_1(m_1-1)(m_1-m_2)C_1(h_1,h_2)}{-m_1(m_1-1)C_1(h_1,h_2)- m_2(m_2-1)C_2(h_1)(z_{\alpha}h_1^{-\gamma})^{m_2-m_1}}+1-m_2\right)
\nonumber\\
%\hspace{-0.3cm}&=&\hspace{-0.3cm}
%\frac{\mu-r}{\sigma^2}\left(\frac{K_1h_1^{1+\gamma(m_1-1)}+K_2h_2^{1+\gamma(m_1-1)}}{K_3h_1^{1+\gamma(m_1-1)}+K_4h_2^{1+\gamma(m_1-1)}}+1-m_2\right)
%\nonumber\\
\hspace{-0.3cm}&=&\hspace{-0.3cm}
\frac{\mu-r}{\sigma^2}
\bigg(\frac{K_1(h_1/h_2)^{1+\gamma(m_1-1)}+K_2}{K_3(h_1/h_2)^{1+\gamma(m_1-1)}+K_4} + 1 - m_2\bigg),
\nonumber
\end{eqnarray}
 where $K_1,K_2,K_3$ and $K_4$ are some constants independent from $x,h_1,h_2$. % with $K_3<0$ and $K_3h_1^{1+\gamma(m_1-1)}+K_4h_2^{1+\gamma(m_1-1)}<0$.
Hence, by $h_1\geq h_2$, it holds that
%\begin{eqnarray}
%\hspace{-0.3cm}&&\hspace{-0.3cm}
%\frac{\mu-r}{\sigma^2}\left(\frac{K_1h_1^{1+\gamma(m_1-1)}+K_2h_2^{1+\gamma(m_1-1)}}{K_3h_1^{1+\gamma(m_1-1)}+K_4h_2^{1+\gamma(m_1-1)}}+1-m_2\right)
%\nonumber\\
%\hspace{-0.3cm}&=&\hspace{-0.3cm}
%\frac{\mu-r}{\sigma^2}\left(\frac{K_1}{K_2}+\left(K_2-\frac{K_1K_4}{K_3}\right)\frac{1}{K_3(h_1/h_2)^{1+\gamma(m_1-1)}+K_4}+1-m_2\right)
%\nonumber\\
%\hspace{-0.3cm}&\leq&\hspace{-0.3cm}
%\frac{\mu-r}{\sigma^2}\left(\left|\frac{K_1}{K_2}+\frac{K_2-\frac{K_1K_4}{K_3}}{K_3+K_4}\right|+1-m_2\right),
%\end{eqnarray}
$$
\frac{K_1(h_1/h_2)^{1+\gamma(m_1-1)}+K_2}{K_3(h_1/h_2)^{1+\gamma(m_1-1)}+K_4} 
$$
is bounded by the maximum/minimum of $\frac{K_1}{K_3}$ and $\frac{K_1+K_2}{K_3+K_4}$,
which yields $\pi^*_x$ is bounded for all $h_1\geq h_2$ and $x_\mathrm{peak}(h_1,h_2) < x \leq x_\mathrm{lavs}(h_1, h_2)$.

For the second case expression, similar calculations as for the first yields that
\begin{eqnarray}\label{eq:pi_x.case2}
%\hspace{-0.3cm}&&\hspace{-0.3cm}
\pi^*_x(x,h_1,h_2)
%\nonumber\\
\hspace{-0.3cm}&=&\hspace{-0.3cm}
\frac{\mu-r}{\sigma^2}\bigg(m_1(m_1-1)(m_1-m_2)C_3(h_1,h_2)f(x,h_1,h_2)^{m_1-2}f_x(x,h_1,h_2)%\right.
\nonumber\\
\hspace{-0.3cm}&&\hspace{1cm}
\left.+\frac{2(\gamma^*-1)(\gamma^*-m_2)}{\kappa^2(\gamma^*-m_1)(\gamma^*-m_2)}f(x,h_1,h_2)^{\gamma^*-2}f_x(x,h_1,h_2)+1-m_2\right)
%\nonumber\\
%\hspace{-0.3cm}&=&\hspace{-0.3cm}
%\frac{\mu-r}{\sigma^2}\left(\left({m_1(m_1-1)(m_1-m_2)C_3(h_1,h_2)f(x,h_1,h_2)^{m_1-2}+\frac{2(\gamma^*-1)}{\kappa^2(\gamma^*-m_1)}f(x,h_1,h_2)^{\gamma^*-2}}\right)\right.
%\nonumber\\
%\hspace{-0.3cm}&&\hspace{1cm}
%\times\left(-m_1(m_1-1)C_3(h_1,h_2)f(x,h_1,h_2)^{m_1-2}- m_2(m_2-1)C_4(h_1,h_2)f(x,h_1,h_2)^{m_2-2}\right.
%\nonumber\\
%\hspace{-0.3cm}&&\hspace{1cm}
%\left.\left.- \frac{2(\gamma^*-1)}{\kappa^2(\gamma^*-m_1)(\gamma^*-m_2)}f(x,h_1,h_2)^{\gamma^*-2}\right)^{-1}+1-m_2\right)
\nonumber\\
\hspace{-0.3cm}&=&\hspace{-0.3cm}
\frac{\mu-r}{\sigma^2}\left(\left({m_1(m_1-1)(m_1-m_2)C_3(h_2)f(x,h_1,h_2)^{m_1-\gamma^*}+\frac{2(\gamma^*-1)}{\kappa^2(\gamma^*-m_1)}}\right)\right.
\nonumber\\
\hspace{-0.3cm}&&\hspace{1cm}
\times\left(-m_1(m_1-1)C_3(h_2)f(x,h_1,h_2)^{m_1-\gamma^*}- m_2(m_2-1)C_4(h_1)f(x,h_1,h_2)^{m_2-\gamma^*}\right.
\nonumber\\
\hspace{-0.3cm}&&\hspace{1cm}
\left.\left.- \frac{2(\gamma^*-1)}{\kappa^2(\gamma^*-m_1)(\gamma^*-m_2)}\right)^{-1}
+1-m_2\right)
\nonumber\\
\hspace{-0.3cm}&=&\hspace{-0.3cm}
 \frac{\mu-r}{\sigma^2}\bigg(\frac{A(x,h_1,h_2)}{\psi(f(x,h_1,h_2)} + 1-m_2\bigg).
\nonumber
\end{eqnarray}
Here, according to the proof of Lemma \ref{lemma: v_convex},  it is easy to see
$$
|A(x,h_1,h_2)| = \big|K_1 h_2^{1+\gamma(m_1-1)}f(x, h_1, h_2)^{m_1-\gamma^*} + K_2\big|
\leq |K_1| + |K_2|,
$$
where $K_1, K_2$ are some constants independent of $x, h_1, h_2$.  The function $y\to \psi(y)$ is given by
$$
\begin{aligned}
\psi(y) :=& -m_1(m_1-1)C_3(h_2)y^{m_1-\gamma^*}- m_2(m_2-1)C_4(h_1)y^{m_2-\gamma^*}\\
&- \frac{2(\gamma^*-1)}{\kappa^2(\gamma^*-m_1)(\gamma^*-m_2)}.
\end{aligned}
$$
%for $y\in[h_1^{-\gamma}, h_2^{-\gamma}]$.
%It follows from $x\in[x_\mathrm{valy}(h_1,h_2), x_\mathrm{peak}(h_1,h_2)]$ that $f(x,h_1,h_2)\in[h_1^{-\gamma},h_2^{-\gamma}]\subset\mathbb{R}_+$ for all $h_1\geq h_2>0$.
%Let us denote a function
%{\blue
To guarantee the boundness of $\pi^*_x(x, h_1, h_2)$, it is sufficient to show $\psi(f(x, h_1, h_2)) \leq K$ for some negative constant $K$ independent of $x, h_1, h_2$.
The derivative of $\psi$:
$$
\psi'(y) = -m_1(m_1-1)(m_1-\gamma^*)C_3(h_1,h_2)y^{m_1-m_2} - m_2(m_2-1)(m_2-\gamma^*)C_4(h_1,h_2)\big)y^{m_2-1-\gamma^*}.
$$
Equation $\psi'(y)=0$ has at most one root $y'$ (if exist), and thus the maximum value of $\psi(y)$ is the max of $\psi(h_1^{-\gamma})$, $\psi(h_2^{-\gamma})$, and $\psi(y')$.
First,
$$
\psi(h_1^{-\gamma}) = K_3 \bigg(\frac{h_2}{h_1}\bigg)^{1+\gamma(m_1-1)}  + K_4  - \frac{2(\gamma^*-1)}{\kappa^2(\gamma^*-m_1)(\gamma^*-m_2)} < 0
$$
for all possible values of $h_1\geq h_2$, where $K_3, K_4$ are some constants independent of $x, h_1, h_2$.
Therefore,
$$
\psi(h_1^{-\gamma}) \leq \max(K_3, 0) + K_4 - \frac{2(\gamma^*-1)}{\kappa^2(\gamma^*-m_1)(\gamma^*-m_2)}< 0.
$$
Similarly, we also get
$$
\psi(h_2^{-\gamma}) < K_3 + \max(K_4, 0) - \frac{2(\gamma^*-1)}{\kappa^2(\gamma^*-m_1)(\gamma^*-m_2)}  < 0.
$$
Finally we consider $\psi(y')$ such that $\psi'(y') = 0$.
According to $\psi'(y') = 0$, we first have
$$
m_2(m_2-1)(m_2-\gamma^*)C_4(h_1,h_2) = -m_1(m_1-1)(m_1-\gamma^*)C_3(h_1,h_2)(y')^{m_1-m_2},
$$
and thus
$$
\begin{aligned}
\psi(y') %&= -m_1(m_1-1)\bigg(1-\frac{m_1-\gamma^*}{m_2-\gamma^*}\bigg)C_3(h_1,h_2)(y')^{m_1-\gamma^*} - \frac{2(\gamma^*-1)}{\kappa^2(\gamma^*-m_1)(\gamma^*-m_2)} \\
&= K_3\bigg(\frac{h_1}{h_2}\bigg)^{\frac{(m_1-\gamma^*)(m_2-\gamma^*)}{(m_1-m_2)(1-\gamma^*)}} - \frac{2(\gamma^*-1)}{\kappa^2(\gamma^*-m_1)(\gamma^*-m_2)} \\
&\leq \max(K_3, 0) - \frac{2(\gamma^*-1)}{\kappa^2(\gamma^*-m_1)(\gamma^*-m_2)} < 0.
\end{aligned}
$$

For the third case expression, one can obtain that
\begin{eqnarray}
\hspace{-0.3cm}&&\hspace{-0.3cm}
\pi^*_x(x,h_1,h_2)
\nonumber\\
\hspace{-0.3cm}&=&\hspace{-0.3cm}
\frac{\mu-r}{\sigma^2}\left(\frac{m_1(m_1-1)(m_1-m_2)C_5(h_2)}{-m_1(m_1-1)C_5(h_1,h_2)- m_2(m_2-1)C_6(h_1,h_2)f(x,h_1,h_2)^{m_2-m_1}}+1-m_2\right)
\nonumber\\
%\hspace{-0.3cm}&\leq&\hspace{-0.3cm}
%\frac{\mu-r}{\sigma^2}\left(\frac{m_1(m_1-1)(m_1-m_2)C_5(h_2)}{-m_1(m_1-1)C_5(h_1,h_2)- m_2(m_2-1)C_6(h_1,h_2)(h_1^{-\gamma})^{m_2-m_1}}+1-m_2\right)
%\nonumber\\
\hspace{-0.3cm}&\leq&\hspace{-0.3cm}
\frac{\mu-r}{\sigma^2}\left(\frac{K_1}{K_2+K_3(h_1/h_2)^{\gamma(m_1-m_2)}}+1-m_2\right),
\end{eqnarray}
for some positive constants $K_1,k_2$ and $K_3$ independent from $h_1,h_2$. Hence, we have $\pi^*_x$ is bounded for all $h_1\geq h_2$ and $x_\mathrm{gloom}(h_1,h_2)\leq x\leq x_\mathrm{valy}(h_1,h_2)$.

 For the fourth and fifth cases, the agent adjusts the reference levels, and the optimal portfolio degenerates into the past three cases, therefore, one can obtain the boundedness of $\pi^*_x$.

We next prove the boundedness of $\pi^*_{h_1}$ and $\pi^*_{h_2}$.
From Lemma \ref{lemma: f_regularity}, we have
\begin{equation}
f_{h_1}(x, h_1, h_2) = -{g_{h_1}(f(x, h_1, h_2), h_1, h_2)}f_x(x,h_1,h_2).\nonumber
\end{equation}
For the first case expression, differentiating \eqref{eq: f_obtain} with respect to $h_1$, we have
\begin{eqnarray}
0
\hspace{-0.3cm}&=&\hspace{-0.3cm}
-m_1C_{1,h_1}(h_1,h_2)f(x,h_1,h_2)^{m_1-1}-m_1(m_1-1)C_1(h_1,h_2)f(x,h_1,h_2)^{m_1-2}f_{h_1}(x,h_1,h_2)
\nonumber\\
\hspace{-0.3cm}&&\hspace{-0.3cm}
-m_2C_{2,h_1}(h_1)f(x,h_1,h_2)^{m_2-1}-m_2(m_2-1)C_2(h_1)f(x,h_1,h_2)^{m_2-2}f_{h_1}(x,h_1,h_2)+\frac{1}{r},\nonumber
\end{eqnarray}
for $ x_\mathrm{peak}(h_1,h_2) < x \leq x_\mathrm{lavs}(h_1, h_2)$.
Plugging this equation into $\pi^*_{h_1}$, we obtain
\begin{eqnarray*}%\label{eq:pi_h1}
%\hspace{-0.3cm}&&\hspace{-0.3cm}
\pi^*_{h_1}(x,h_1,h_2)
%\nonumber\\
\hspace{-0.3cm}&=&\hspace{-0.3cm}
\frac{\mu-r}{\sigma^2}\left(
m_1(m_1-1)(m_1-m_2)C_1(h_1,h_2)f(x,h_1,h_2)^{m_1-2}f_{h_1}(x,h_1,h_2)\right.
\nonumber\\
\hspace{-0.3cm}&&\hspace{1cm}
\left.
+m_1(m_1-m_2)C_{1,h_1}(h_1,h_2)f(x,h_1,h_2)^{m_1-1}+\frac{m_2-1}{r}
\right)
\nonumber\\
\hspace{-0.3cm}&=&\hspace{-0.3cm}
\frac{\mu-r}{\sigma^2}\left(-m_1(m_1-1)(m_1-m_2)C_1(h_1,h_2)f(x,h_1,h_2)^{m_1-2}f_{x}(x,h_1,h_2)\right.
\nonumber\\
\hspace{-0.3cm}&&\hspace{0.3cm}
\left.
\times g_{h_1}(f(x,h_1,h_2),h_1,h_2)+m_1(m_1-m_2)C_{1,h_1}(h_1,h_2)f(x,h_1,h_2)^{m_1-1}+\frac{m_2-1}{r}\right).
\nonumber\\
\end{eqnarray*}
Note that %function
$$
\begin{aligned}
|C_1(h_1,h_2)f(x, h_1, h_2)^{m_1-2}f_x(x,h_1,h_2)| &\leq K_1, \\
\end{aligned}
$$
according to the proof of boundedness of $\frac{\partial \pi^*(x, h_1, h_2)}{\partial x}$, it follows that
$$
\begin{aligned}
&|g_{h_1}(f(x, h_1, h_2), h_1, h_2)| \\
=& \bigg|-m_1C_{1, h_1}(h_1, h_2) f(x, h_1, h_2)^{m_1-1} - m_2C_{2, h_1}(h_1)f(x, h_1, h_2)^{m_2-1} + \frac{1}{r}\bigg|
\leq K_2,
\end{aligned}
$$
by the expression of $C_1(h_1,h_2), C_2(h_1)$ and $f(x,h_1,h_2) \in [z_\alpha h_1^{-\gamma}, h_1^{-\gamma}]$,
and
$$
C_{1, h_1}(h_1,h_2)f(x, h_1, h_2)^{m_1-1} \leq K_3,
$$
where $K_1, K_2, K_3$ are constants independent of $x, h_1, h_2$. Therefore, we obtain that $\pi^*_{h_1}(x,h_1,h_2)$ is bounded in this case of expression.

For the second case expression, similar calculations as for the first yields that
\begin{eqnarray*}
%\hspace{-0.3cm}&&\hspace{-0.3cm}
\pi^*_{h_1}(x,h_1,h_2)
%\nonumber\\
\hspace{-0.3cm}&=&\hspace{-0.3cm}
\frac{\mu-r}{\sigma^2}\left(m_1(m_1-1)(m_1-m_2)C_3(h_2)f(x,h_1,h_2)^{m_1-2}f_{h_1}(x,h_1,h_2)
\right.
\nonumber\\
\hspace{-0.3cm}&&\hspace{1cm}
\left.+\frac{2(\gamma^*-1)(\gamma^*-m_2)}{\kappa^2(\gamma^*-m_1)(\gamma^*-m_2)}f(x,h_1,h_2)^{\gamma^*-2}f_{h_1}(x,h_1,h_2)\right)
\nonumber\\
\hspace{-0.3cm}&=&\hspace{-0.3cm}
\frac{\mu-r}{\sigma^2}\left(-m_1(m_1-1)(m_1-m_2)C_3(h_2)f(x,h_1,h_2)^{m_1-2}f_x(x,h_1,h_2)
\right.
\nonumber\\
\hspace{-0.3cm}&&\hspace{1cm}
\left.+\frac{2(\gamma^*-1)(\gamma^*-m_2)}{\kappa^2(\gamma^*-m_1)(\gamma^*-m_2)}f(x,h_1,h_2)^{\gamma^*-2}\right){g_{h_1}(f(x, h_1, h_2), h_1, h_2)},
\end{eqnarray*}
for $x_\mathrm{valy}(h_1,h_2) \leq x \leq x_\mathrm{peak}(h_1,h_2)$.
Note that %the function
$$
\begin{aligned}
|g_{h_1}(f(x, h_1, h_2),h_1,h_2)|
%=& \bigg|-m_1C_{3, h_1}(h_1, h_2) f(x, h_1, h_2)^{m_1-1} - m_2C_{4, h_1}(h_1, h_2)f(x, h_1, h_2)^{m_2-1}\bigg|\\
=& |m_2 C_{4, h_1}(h_1)f(x, h_1, h_2)^{m_2-1}|
\leq  K,
\end{aligned}
$$
where $K$ is a constant independent of $x, h_1, h_2$ due to the fact that $f(x, h_1, h_2)\geq h_1^{-\gamma}$.
From the proof of the boundedness of $\pi^*_x$ for the second case, one knows
$$-m_1(m_1-1)(m_1-m_2)C_3(h_2)f(x,h_1,h_2)^{m_1-2}f_x(x,h_1,h_2)+\frac{2(\gamma^*-1)(\gamma^*-m_2)}{\kappa^2(\gamma^*-m_1)(\gamma^*-m_2)}f(x,h_1,h_2)^{\gamma^*-2}$$ is also bounded, implying that $\pi^*_{h_1}$ is bounded for all $h_1\geq h_2$ and $x_\mathrm{valy}(h_1,h_2) \leq x \leq x_\mathrm{peak}(h_1,h_2)$.

For the third case expression, one can obtain that
\begin{eqnarray}
\hspace{-0.3cm}&&\hspace{-0.3cm}
\pi^*_{h_1}(x,h_1,h_2)
\nonumber\\
\hspace{-0.3cm}&=&\hspace{-0.3cm}
\frac{\mu-r}{\sigma^2}\left(m_1(m_1-1)(m_1-m_2)C_5(h_2)f(x,h_1,h_2)^{m_1-2}f_{h_1}(x,h_1,h_2)\right)
\nonumber\\
\hspace{-0.3cm}&=&\hspace{-0.3cm}
\frac{\mu-r}{\sigma^2}\left(-m_1(m_1-1)(m_1-m_2)C_5(h_2)f(x,h_1,h_2)^{m_1-2}f_{x}(x,h_1,h_2)g_{h_1}(f(x,h_1,h_2),h_1,h_2)\right),\nonumber
\end{eqnarray}
for $x_\mathrm{gloom}(h_1,h_2) \leq x < x_\mathrm{valy}(h_1, h_2)$. Noting that the function
\begin{eqnarray}
|g_{h_1}(y,h_1,h_2)| = |-m_2C_{6,h_1}(h_1,h_2)y^{m_2-1}| = K (yh_1^{\gamma})^{m_2-1}\leq K, \nonumber
\end{eqnarray}
%is positive and decreasing with respect to $y$, it holds that
%\begin{eqnarray}
%g_{h_1}(f(x,h_1,h_2),h_1,h_2)
%\hspace{-0.3cm}&\leq&\hspace{-0.3cm}
%g_{h_1}(h_2^{-\gamma},h_1,h_2) ={\blue K \bigg(\frac{h_1}{h_2}\bigg)^{\gamma(m_2-1)} \},
%\nonumber%\\
%\hspace{-0.3cm}&\leq&\hspace{-0.3cm}
%\left|m_2\left(\frac{(1-\gamma^*)(1+\gamma(m_2-1))}{(m_1-m_2)(m_2-\gamma^*)}\left(\frac{m_1(\alpha\delta-1)}{\delta z_{\alpha}^{m_2}}+\frac{m_1-1}{r z_{\alpha}^{m_2-1}}\right)\right.\right.
%\nonumber\\
%\hspace{-0.3cm}&&\hspace{1cm}
%\left.\left.-\frac{2(\gamma^*-1)(1+\gamma(m_2-1))}{\kappa^2(m_1-m_2)m_2(m_2-1)(m_2-\gamma^*)}\right)\right|.
%\end{eqnarray}
 where $K$ is a constant independent of $x, h_1, h_2$.
From the proof of the boundedness of $\pi^*_x$ for the third case, one knows
$$m_1(m_1-1)(m_1-m_2)C_5(h_2)f(x,h_1,h_2)^{m_1-2}f_x(x,h_1,h_2)$$ is also bounded for all $h_1\geq h_2$ and for $x_\mathrm{gloom}(h_1,h_2) \leq x < x_\mathrm{valy}(h_1, h_2)$, implying $\pi^*_{h_1}$ is bounded for all $h_1\geq h_2$ and for $x_\mathrm{gloom}(h_1,h_2) \leq x < x_\mathrm{valy}(h_1, h_2)$.

%For the fourth and fifth cases, one can obtain the boundedness of $\pi^*_{h_1}$ by a similar method as used in the first three cases.
Similarly, one can also obtain the boundedness of $\pi^*_{h_2}$, which completes the proof.

\subsection{Proof of Theorem \ref{thm: sensitivity}}\label{sec:proof_thm_sensitivy}
The item (i) can be verified directly. In what follows, we prove item (ii). Since $\gamma^*=-(1-\gamma)/\gamma$, we have $\gamma^*\to 1$ as $\gamma\to \infty$. Then, for $h_1\geq h_2\geq 0$ and $\gamma^*=1$,  it follows from Corollary \ref{cor: main_res} that
\begin{align}\label{eq:c-bx-ux}
c^*(x_\mathrm{lavs}(h_1, h_2),h_1,h_2)=h_1,\quad c^*(x_\mathrm{gloom}(h_1, h_2),h_1,h_2)=h_2,
\end{align}
and
\begin{align}
x_\mathrm{lavs}(h_1, h_2)&=-m_1C_1(h_1, h_2)(z_\alpha h_1^{-\gamma})^{m_1-1} - m_2C_2( h_2)(z_\alpha h_1^{-\gamma})^{m_2-1} + \frac{h_1}{r}=\frac{h_1}{r},\\
x_\mathrm{gloom}(h_1, h_2) &= -m_1C_5(h_2)(z_\beta h_2^{-\gamma})^{m_1-1} - m_2C_6(h_1, h_2)(z_\beta h_2^{-\gamma})^{m_2-1} + \frac{h_2}{r}=\frac{h_2}{r}.\label{eq:bx-ux}
\end{align}
By using \eqref{eq:c-bx-ux}-\eqref{eq:bx-ux}, we get the desired result.

\subsection{Proof of Proposition \ref{prop:asymptotic}}\label{sec:proof_prop_asymptotic}
We only prove item (i) here, and item (ii) can be showed in a similar fashion. For $h_1\geq h_2\geq 0$, as $x\to \infty$, we have $x> x_\mathrm{lavs}(h_1,h_2)$ together with $y(x)\to 0$, where the function $x\to y(x)$ is determined by
\begin{align}\label{eq:x-y-infinity}
x =& -m_1C_1\big(\bh_1\big(y\big), h_2\big)y^{m_1-1}- m_2C_2(\bh_1\big(y\big)\big)y^{m_2-1}+\frac{\bh_1\big(y\big)}{r}.
\end{align}
Then it follows from \eqref{eq:x-y-infinity} and Proposition \ref{prop: solu_dual_value} that
\begin{align}\label{asymptotic-x-infinity}
x &= -\frac{2(1-\gamma^*)}{\kappa^2(m_1-m_2)(m_1-1)(m_1-\gamma^*)}z_\alpha^{\frac{1}{\gamma}+(m_1-1)}y^{-\frac{1}{\gamma}} -m_1C_3(h_2)y^{m_1-1}\nonumber\\
&\quad- \frac{m_2(1-\gamma^*)}{(m_1-m_2)(m_2-\gamma^*)}\bigg(\frac{m_1(\alpha\delta-1)}{\delta}z_\alpha^{-m_2} + \frac{m_1-1}{r}z_\alpha^{1-m_2}\bigg)z_\alpha^{\frac{1}{\gamma}+(m_2-1)}y^{-\frac{1}{\gamma}} +\frac{1}{r}z_\alpha^{\frac{1}{\gamma}}y^{-\frac{1}{\gamma}} \nonumber\\
&=z_\alpha^{\frac{1}{\gamma}-1}y^{-\frac{1}{\gamma}}\Bigg( -\frac{2(1-\gamma^*)}{\kappa^2(m_1-m_2)(m_1-1)(m_1-\gamma^*)}z_\alpha^{m_1}+\Bigg(\frac{1}{r}-\frac{m_2(1-\gamma^*)(m_1-1)}{r(m_1-m_2)(m_2-\gamma^*)}\Bigg)z_{\alpha}\nonumber\\
&\quad- \frac{m_1m_2(1-\gamma^*)(\alpha\delta-1)}{\delta(m_1-m_2)(m_2-\gamma^*)}\Bigg) -m_1C_3( h_2)y^{m_1-1}.
\end{align}
Using \eqref{eq: z_alpha_beta}, we obtain
\begin{align}\label{eq:z_alpha_infinity}
-\frac{m_1m_2(1-\gamma^*)(\alpha\delta-1)}{\delta(m_1-m_2)(m_2-\gamma^*)} =\frac{2(1-\gamma^*)}{\kappa^2 (m_1-1)(m_1-m_2)(m_2-\gamma^*)}z_\alpha^{m_1} + \frac{m_1(m_2-1)(1-\gamma^*)}{r(m_1-m_2)(m_2-\gamma^*)}z_\alpha.
\end{align}
Plugging this equation into \eqref{asymptotic-x-infinity}, one gets that
\begin{align}\label{eq:asymptotic-x-infinity}
x&=\frac{1}{(m_2-\gamma^*)}\Bigg( \frac{2(1-\gamma^*)}{\kappa^2(m_1-1)(m_1-\gamma^*)}z_\alpha^{m_1-1}+\frac{m_2-1}{r}\Bigg)z_\alpha^{\frac{1}{\gamma}}y^{-\frac{1}{\gamma}}-m_1C_3( h_2)y^{m_1-1}.
\end{align}
Thanks to Theorem \ref{thm: verification} and \eqref{eq:z_alpha_infinity}, for $x> x_\mathrm{lavs}(h_1,h_2)$, i.e., $y(x) < z_\alpha h_1^{-\gamma}$, the optimal portfolio function is give by
\begin{align}\label{eq:asymptotic-pi-infinity}
\pi^\dag(y, h_1, h_2)
&=\frac{\mu-r}{\sigma^2}\big(m_1(m_1-1)C_1\big(\bh_1(y), h_2\big)y^{m_1-1} + m_2(m_2-1)C_2\big(\bh_1(y)\big)y^{m_2-1}\big)\nonumber\\
%&=\frac{\mu-r}{\sigma^2}(1-\gamma^*)z_\alpha^{\frac{1}{\gamma}-1}y^{-\frac{1}{\gamma}}\Bigg( \frac{2}{\kappa^2(m_1-m_2)(m_1-\gamma^*)}z_\alpha^{m_1}+\frac{m_2(m_1-1)(m_2-1)}{r(m_1-m_2)(m_2-\gamma^*)}z_{\alpha}\Bigg)\nonumber\\
%&\quad+\frac{m_1m_2(m_2-1)(\alpha\delta-1)}{\delta(m_1-m_2)(m_2-\gamma^*)}\Bigg) +\frac{\mu-r}{\sigma^2}m_1(m_1-1)C_3( h_2)y^{m_1-1}\nonumber\\
&=\frac{\mu-r}{\sigma^2\gamma(m_2-\gamma^*)}\Bigg( \frac{2(1-\gamma^*)}{\kappa^2(m_1-1)(m_1-\gamma^*)}z_\alpha^{m_1-1}+\frac{m_2-1}{r}\Bigg) z_\alpha^{\frac{1}{\gamma}}y^{-\frac{1}{\gamma}}\nonumber\\
&\quad +\frac{\mu-r}{\sigma^2}m_1(m_1-1)C_3( h_2)y^{m_1-1}.
\end{align}
 By \eqref{eq:asymptotic-x-infinity} and \eqref{eq:asymptotic-pi-infinity}, we deduce that
\begin{align*}
\bar{\pi}^*_{\mathrm{ratio}}=\lim_{x\to\infty}\frac{\pi^*(x,h_1,h_2)}{x}=\lim_{y\to 0}\frac{\pi^\dag(y,h_1,h_2)}{x(y)}=\frac{\mu-r}{\sigma^2 \gamma}.
\end{align*}
Note that for $x > x_\mathrm{lavs}(h_1, h_2)$, i.e., $y(x) < z_\alpha h_1^{-\gamma}$, the optimal consumption function $c^\dag= \bh_1(y)= (y/z_\alpha)^{-1/\gamma}$. This yields that
\begin{align*}
\bar{c}^*_{\mathrm{ratio}}=\lim_{x\to\infty}\frac{c^*(x,h_1,h_2)}{x}=\lim_{y\to 0}\frac{c^\dag(y,h_1,h_2)}{x(y)}&=\frac{1}{ \frac{2(1-\gamma^*)}{\kappa^2(m_1-1)(m_1-\gamma^*)(m_2-\gamma^*)}z_\alpha^{m_1-1}+\frac{m_2-1}{r(m_2-\gamma^*)}}\\
%&=\frac{r\kappa^2(m_1-1)(m_1-\gamma^*)(m_2-\gamma^*)}{2r(1-\gamma^*)z_{\alpha}^{m_1-1}+\kappa^2(m_1-1)(m_2-1)(m_1-\gamma^*)}\\
&=\frac{\kappa^2(m_1-1)(m_1-\gamma^*)(m_2-\gamma^*)}{2((1-\gamma^*)z_{\alpha}^{m_1-1}-m_1+\gamma^*))}.
\end{align*}
Differentiating  with respect to $\alpha$ on both side of \eqref{eq: z_alpha_beta}, we have
\begin{align*}
\frac{d z_\alpha}{d\alpha}=\frac{2m_2(1-z_{\alpha}^{m_1-1})}{\kappa^2(m_1-1)}<0,
\end{align*}
which implies that $\alpha\to z_{\alpha}$ is decreasing. Since  $m_1 > 1 > \gamma^* > m_2$, the function $\alpha \to \bar{c}^*_{\mathrm{ratio}}$ is decreasing. Especially, when $\alpha=0$, $z_{\alpha}=1$, then it is not hard to verify that
\begin{align*}
\bar{c}^*_{\mathrm{ratio}}&=-\frac{\kappa^2}{2}(m_1-\gamma^*)(m_2-\gamma^*)=\frac{\delta}{\gamma}-(1-\gamma)\left(\frac{(\mu-r)^2}{2\sigma^2\gamma^2}+\frac{r}{\gamma}\right),
\end{align*}
which is exactly the optimal consumption strategy of the classical Merton’s problem in \cite{Mert1971JET}. Thus, we get the desired result of item (i).

\subsection{Proof of Proposition \ref{prop:portfolio}}\label{sec:proof_prop_portfolio}
 We  claim that $C_3(h_2)\leq 0,C_4(h_1)\geq 0$ for all $h_1\geq h_2\geq 0$. In view of $z_\alpha \in (0, 1-\alpha\delta]$, $z_\beta \in [1+\beta\delta, +\infty)$ and \eqref{eq: C_solu}, we have that
\begin{align*}
\frac{dC_3}{dz_{\beta}}&=-\frac{2(1-\gamma^*)h_2^{1+\gamma(m_1-1)}z_{\beta}^{-1-m_1}}{\kappa^2(m_1-m_2)(m_1-\gamma^*)}(z_{\beta}-(\beta\delta+1))<0,\\
\frac{dC_4}{dz_{\alpha}}&=\frac{2(1-\gamma^*)h_1^{1+\gamma(m_2-1)}z_{\alpha}^{-1-m_2}}{\kappa^2(m_1-m_2)(\gamma^*-m_2)}((1-\alpha\delta)-z_{\alpha})>0,
\end{align*}
which yields that $C_3(h_2)\leq C_3(h_2)|_{z_{\beta}=1}\leq 0$ and  $C_4(h_1)\leq C_4(h_1)|_{z_{\alpha}=1}\leq 0$ for all $h_1\geq h_2\geq 0$. Then we show $\pi^*_{\mathrm{ratio}}(x,h_1,h_2)\leq \frac{\mu-r}{\sigma^2\gamma}$ in different regions respectively.

 For the case with  $x> x_\mathrm{lavs}(h_1,h_2)$, using  \eqref{eq:asymptotic-x-infinity} and \eqref{eq:asymptotic-pi-infinity} again, we have that
\begin{align*}
\pi^*_{\mathrm{ratio}}(x,h_1,h_2)=\frac{\mu-r}{\sigma^2\gamma}+\frac{\mu-r}{\sigma^2}\frac{m_1(m_1-\gamma^*)C_3( h_2)y^{m_1-1}}{x}\leq \frac{\mu-r}{\sigma^2\gamma}.
\end{align*}
For the case with $x_\mathrm{peak}(h_1,h_2) < x \leq x_\mathrm{lavs}(h_1, h_2)$, i.e., $z_\alpha h_1^{-\gamma} \leq y < h_1^{-\gamma}$, we have that
\begin{align*}
&\pi^*_{\mathrm{ratio}}(x,h_1,h_2)\leq \frac{\mu-r}{\sigma^2\gamma}
\Leftrightarrow \pi^\dag(y, h_1, h_2)\leq -\frac{\mu-r}{\sigma^2\gamma} v_y(y,h_1,h_2)\\
&\Leftrightarrow F(y,h_1,h_2):=m_1(m_1-\gamma^*)C_1(h_1, h_2)y^{m_1-1} + m_2(m_2-\gamma^*)C_2(h_1)y^{m_2-1}-\frac{h_1}{r\gamma}\leq 0.
\end{align*}
Note that
\begin{align*}
\frac{\partial F(y,h_1,h_2)}{\partial y}=\left(m_1(m_1-1)(m_1-\gamma^*)C_1(h_1, h_2)y^{m_1-m_2} + m_2(m_2-1)(m_2-\gamma^*)C_2(h_1)\right)y^{m_2-2},
\end{align*}
thus we know that one of the following cases holds: (i) $F(y,h_1,h_2)\geq 0$ for all $z_\alpha h_1^{-\gamma} \leq y < h_1^{-\gamma}$; (ii) $F(y,h_1,h_2)\leq 0$ for all $z_\alpha h_1^{-\gamma} \leq y < h_1^{-\gamma}$; or (iii) there exists $y_0\in (z_\alpha h_1^{-\gamma}, h_1^{-\gamma})$ such that $F(y,h_1,h_2)\leq 0$ for $y\in(z_\alpha h_1^{-\gamma},y_0)$ and $F(y,h_1,h_2)\geq 0$ for $y\in(y_0, h_1^{-\gamma})$. This yields that
$F(y,h_1,h_2)\leq \max\{F(h_1^{-\gamma},h_1,h_2),F(z_\alpha h_1^{-\gamma},h_1,h_2)\}$. By direct calculation, it follows that
\begin{align*}
F(z_\alpha h_1^{-\gamma},h_1,h_2)\leq 0,\quad F(h_1^{-\gamma},h_1,h_2)\leq \frac{z_{\alpha}^{1-m_2}-1}{r\gamma}+\frac{2(1-\gamma^*)}{\kappa^2(m_1-m_2)(m_1-1)}\left(1-z_{\alpha}^{m_1-m_2}\right).
\end{align*}
It is not hard to verify that the function $\frac{z_{\alpha}^{1-m_2}-1}{r\gamma}+\frac{2(1-\gamma^*)}{\kappa^2(m_1-m_2)(m_1-1)}\left(1-z_{\alpha}^{m_1-m_2}\right)$ is increasing with respect to $z_{\alpha}$. Since $z_{\alpha}\leq 1$,  we have $F(h_1^{-\gamma},h_1,h_2)\leq 0$ which implies $\pi^*_{\mathrm{ratio}}(x,h_1,h_2)\leq \frac{\mu-r}{\sigma^2\gamma}$ for $z_\alpha h_1^{-\gamma} \leq y < h_1^{-\gamma}$.

For the case with $x_\mathrm{valy}(h_1,h_2) \leq x \leq x_\mathrm{peak}(h_1,h_2)$, i.e., $ h_1^{-\gamma} \leq y < h_2^{-\gamma}$, we have that
\begin{align*}
\pi^*_{\mathrm{ratio}}(x,h_1,h_2)\leq \frac{\mu-r}{\sigma^2\gamma}
&\Leftrightarrow \pi^\dag(y, h_1, h_2)\leq -\frac{\mu-r}{\sigma^2\gamma} v_y(y,h_1,h_2)\\
&\Leftrightarrow m_1(m_1-\gamma^*)C_3(h_2)y^{m_1-1} + m_2(m_2-\gamma^*)C_4(h_1)y^{m_2-1}\leq 0,
\end{align*}
which obviously holds as $C_3(h_2)\leq 0,C_4(h_1)\leq 0$. The proof of the case with $ x_\mathrm{gloom}(h_1,h_2) \leq x < x_\mathrm{valy}(h_1, h_2)$ and $x < x_\mathrm{gloom}(h_1, h_2)$ is similar, thus we omit it here. Putting all the pieces together, we get the desired result.

\section*{Acknowledgements}
Y. Huang, K. Yan and Q. Zhang are supported by the Hong Kong Polytechnic University research grant under no. P0045654.

\bibliographystyle{plain}
{\small
\bibliography{references}
}

\end{document}